\documentclass[11pt, twoside]{article}

\usepackage{amsmath}
\usepackage{amsthm}
\usepackage{amssymb}
\usepackage{mathrsfs}  
\usepackage[all]{xy}
\usepackage{multirow}
\usepackage{caption}
\usepackage{float}
\usepackage{multicol}
\usepackage{soul}
\usepackage{xcolor}
\usepackage{enumerate}
\usepackage{tikz-cd}

\usepackage{setspace}
\usepackage{geometry}
\usepackage{color}

\newcommand{\bF}{{\bf F}}
\newcommand{\matV}{\ensuremath{\mathcal{V}}}
\newcommand{\matGM}{\ensuremath{\mathcal{RM}}}
\newcommand{\matGS}{\ensuremath{\mathcal{RS}}}

\newcommand{\lfis}{{\bf LFI}s}

\newcommand{\mbc}{{\bf mbC}}

\newcommand{\mbCcl}{{\bf mbCcl}}

\newcommand{\cila}{{\bf Cila}}

\newcommand{\A}{\ensuremath{\mathcal{A}}}
\newcommand{\B}{\ensuremath{\mathcal{B}}}
\newcommand{\matM}{\ensuremath{\mathcal{M}}}
\newcommand{\matF}{\ensuremath{\mathcal{F}}}

\newcommand{\MP}{\textbf{MP}}

\newcommand{\kax}{{\bf Ax1}}
\newcommand{\axTrans}{{\bf Ax2}}
\newcommand{\axed}{{\bf Ax3}}
\newcommand{\axeea}{{\bf Ax4}}
\newcommand{\axeeb}{{\bf Ax5}}
\newcommand{\axouda}{{\bf Ax6}}
\newcommand{\axoudb}{{\bf Ax7}}
\newcommand{\axoue}{{\bf Ax8}}
\newcommand{\axouimp}{{\bf Ax9}}
\newcommand{\axtnd}{{\bf Ax10}}
\newcommand{\axexp}{{\bf bc}}

\newcommand{\axcf}{{\bf Ax11}}
\newcommand{\axdc}{{\bf dc}}
\newcommand{\axp}{{\bf P}}

\newcommand{\imp}{\to}

\newcommand{\cons}{\ensuremath{{\circ}}}

\newcommand{\sneg}{\ensuremath{{\sim}}}

\newtheorem{theorem}{Theorem}[section]
\newtheorem{lemma}[theorem]{Lemma}
\newtheorem{prop}[theorem]{Proposition}
\newtheorem{coro}[theorem]{Corollary}

\newtheorem{defi}[theorem]{Definition}
\newtheorem{remark}[theorem]{Remark}

\newtheorem{remarks}[theorem]{Remarks}

\usepackage{authblk}

\title{\textbf{Restricted swap structures for da Costa's\\ $C_n$ and their category}}
\author{Coniglio, Marcelo E.\thanks{coniglio@unicamp.br} }

\author{Toledo, Guilherme V.\thanks{guivtoledo@gmail.com}}

\affil{Institute of Philosophy and the Humanities - IFCH and\\
Centre for Logic, Epistemology and The History of Science - CLE\\
University of Campinas - Unicamp\\
Campinas, SP, Brazil}

\usepackage{fancyhdr}
\providecommand{\keywords}[1]{\textbf{\textit{Keywords:}} #1}

\begin{document}

\setcounter{page}{1}     

\maketitle

\begin{abstract}
In a previous article we introduced the concept of restricted Nmatrices (in short, RNmatrices), which generalize Nmatrices in the following sense: a RNmatrix is a Nmatrix together with a {\em subset} of valuations over it, from which  the consequence relation is defined. Within this semantical framework we have characterized each paraconsistent logic $C_{n}$ in the hierarchy of da Costa by means of a $(n+2)$-valued RNmatrix, which also provides a relatively simple decision procedure for each calculus (recalling that $C_1$ cannot be characterized by a single finite Nmatrix). In this paper we extend such RNmatrices for $C_{n}$ by means of what we call {\em restricted swap-structures} over arbitrary Boolean algebras, obtaining so a class of non-deterministic semantical structures which characterizes da Costa's systems. We give a brief algebraic and combinatorial description of the elements of the underlying RNmatrices. Finally, by presenting a notion of category of RNmatrices, we show that the category of RNmatrices for $C_{n}$ is in fact isomorphic to the category of non-trivial Boolean algebras.
  
\end{abstract}


\section{Introduction}

Newton C. A. da Costa defined, in 1963 (\cite{dC63}), a hierarchy of logical systems $C_{n}$ with the aim of formalize paraconsistent reasoning. This was the first systematic approach to paraconsistency, recalling that the first formal  paraconsistent system was the {\em Discussive} or {\em Discursive Logic} introduced in 1948 by  Stanis\l aw Ja\'{s}kowski (\cite{jas:48,jas:49}). Indeed, da Costa's logics started a revolution in the field of non-classical logic, motivating the introduction of several new semantical frameworks to deal with their intrinsically difficult nature. Among these we can mention bivaluations (\cite{dCA:77, lop:alv:80}), Fidel structures (\cite{fid:77}), Nmatrices (\cite{avr:lev:01, avr:lev:05}), and swap structures (\cite{CC16}).

Nmatrices, considered first in the literature by Rescher and Ivlev (\cite{res:62,ivl:73,ivl:85,ivl:88,ivl:13}), generalize logical matrices by replacing an algebra with a multialgebra. By combining a different generalization of logical matrices proposed by Piochi (\cite{Piochi,Piochi2}), we have defined RNmatrices in a previous study (\cite{ConTol}). As mentioned there, RNmatrices were first considered by  Kearns with the aim of providing a new semantics for normal modal logics different to the standard Kripke semantics (\cite{kear:81}). Kearns' RNmatrices  approach to modal logics was afterwards considered in~\cite{con:far:per:15,con:far:per:16, omo:sku:16, OS:20}. In~\cite{gratz:21}, a decision procedure for several normal modal logics was obtained by Gr\"atz by refining the original RNmatrices originally proposed by Kearns. RNmatrices were also considered by Pawlowski and Urbaniak in the context of logics of informal provability (\cite{Pawlowski, paw:urb:18}). In~\cite{ConTol} we also  show how several different semantical methodologies may be recast as RNmatrices, including Fidel and swap structures, bivaluations, static Nmatrices (\cite{AK:05}), and PNmatrices (\cite{Baaz:13, CM:19}).

The most significative part of \cite{ConTol}, however, was the construction of finite ($n+2$-valued), manageable RNmatrices $\mathcal{RM}_{C_{n}}$ capable of characterize $C_{n}$. In particular, $\mathcal{RM}_{C_{n}}$ are a perfect example of the cases in which RNmatrices induce a row-branching truth-table where one can algorithmically select those rows that correspond to unwanted homomorphisms, leading therefore to a decision method for its respective logic.
We have, furthermore, provided a tableaux semantics built upon $\mathcal{RM}_{C_{n}}$ which, although not the first tableaux semantics for $C_{n}$ (\cite{d'ott:cast:06}), are very intuitive since are generated by  the corresponding  RNmatrices in a very natural way.

The construction of $\mathcal{RM}_{C_{n}}$ involves taking a bivaluation $\mathsf{b}$ for $C_{n}$ and associating to a formula $\alpha$ the $n+1$-tuple $(\mathsf{b}(\alpha), \mathsf{b}(\neg\alpha), \mathsf{b}(\alpha^{1}), \ldots  , \mathsf{b}(\alpha^{n-1}))$ (called a {\em snapshot}), a construction reminiscent of that of swap structures. This is not without reason, given the latter in fact motivated the former, but one is left to wonder whether this construction, carried over entirely upon the two-valued Boolean algebra (over which one defines bivaluations), may be generalized to any non-trivial Boolean algebras, as is the case with many swap structures: the answer is yes. This is important, first of all, for model-theoretical reasons, seeing that we present a class of non-isomorphic models for $C_{n}$ of varying complexities. However, this is also relevant as it suggests how to approach working categorically over $C_{n}$, given that the aforementioned class of models forms a nicely-behaving category.

We start this article with some preliminaries in Section~\ref{Preliminaries}, explicitly defining RNmatrices, da Costa's Calculi $C_{n}$ and the RNmatrices $\mathcal{RM}_{C_{n}}$. In Section~\ref{Bval} we begin by defining $\mathcal{B}$-valuations, generalizing bivaluations for an arbitrary Boolean algebra $\mathcal{B}$, and in Section~\ref{RS} we construct the expansions of $\mathcal{RM}_{C_{n}}$ by $\mathcal{B}$, aptly named $\mathcal{RM}_{C_{n}}^{\mathcal{B}}$, which we then show to characterize $C_{n}$. Section~\ref{Counting_snapshots} uses some elementary combinatorial methods to count the snapshots of the Nmatrix underlying $\mathcal{RM}_{C_{n}}^{\mathcal{B}}$, as well as its designated and classically-behaving elements, to show, specially in the finite case, how these objects have a rich structure, from an algebraic standpoint. Section~\ref{Category} suggests how one could approach the general problem of defining a category for an arbitrary class of RNmatrices and proceeds to apply this very definitions to the class of $\mathcal{RM}_{C_{n}}^{\mathcal{B}}$, for any non-trivial Boolean algebra $\mathcal{B}$. In what is a surprising result, we show that the restrictions on morphisms of categories of RNmatrices imply that the resulting category in $C_{n}$'s case is actually isomorphic to the category of non-trivial Boolean algebras. Some final considerations, as well as future works, are given in Section~\ref{FinRem}.

\section{Preliminaries}\label{Preliminaries}

\subsection{Restricted non-deterministic matrices} \label{defRNmat}

A (propositional) {\em signature} is a family $\Theta=\{\Theta_n\}_{n\in\mathbb{N}}$ of pairwise disjoint sets, elements of $\Theta_n$ being called {\em $n$-ary connectives}.
The $\Theta$-algebra freely generated by a set $\matV=\{p_{n}\}_{n\in\mathbb{N}}$ of propositional variables is denoted, here, by ${\bF}(\Theta,\matV)$, and its universe, the set of {\em formulas} over $\Theta$, by $F(\Theta, \matV)$. Endomorphisms of ${\bf F}(\Theta, \mathcal{V})$ (that is, homomorphisms from ${\bf F}(\Theta, \mathcal{V})$ to itself) are called {\em substitutions}.


Fixed a signature $\Theta$, a {\em logical matrix} is a pair $\mathcal{M}=(\mathcal{A}, D)$ such that: $\mathcal{A}$ is a $\Theta$-algebra; and $D$ is a subset of the universe of $\mathcal{A}$. A logical matrix $\mathcal{M}$ defines a consequence operator over $F(\Theta, \matV)$ such that $\Gamma \vDash_{\mathcal{M}}\varphi$ iff, for every homomorphism $\nu:{\bF}(\Theta, \matV)$, $\nu[\Gamma]\subseteq D$\footnote{Given a function $f:X \to Y$ and a set $Z \subseteq X$, $f[Z]$ denotes $\{f(x) \ : \ x \in Z\}$.} implies $\nu(\varphi)\in D$. Given a class $\mathbb{M}$ of logical matrices, $\Gamma \vDash_{\mathbb{M}}\varphi$ iff $\Gamma \vDash_{\mathcal{M}}\varphi, \forall\mathcal{M} \in \mathbb{M}$. 

\begin{defi} \label{defNmatrix}
Fix a signature $\Theta$.
\begin{enumerate}
\item For a set $A$, a pair $\mathcal{A}=(A,\{\sigma_{\mathcal{A}}\}_{\sigma\in\Theta})$ is said to be a {\em $\Theta$-multialgebra} if, for any $n$-ary $\sigma$, $\sigma_{\mathcal{A}}$ is a function from $A^{n}$ to $\wp(A)\setminus\{\emptyset\}$; $A$ is called the universe of $\mathcal{A}$.
\item Given $\Theta$-multialgebras $\mathcal{A}$ and $\mathcal{B}$, with universes $A$ and $B$, a {\em homomorphism} between $\mathcal{A}$ and $\mathcal{B}$ is a function $f:A\rightarrow B$ satisfying, for any $n$-ary $\sigma$ and elements $a_{1}, \ldots , a_{n}$ of $A$,
$f[\sigma_{\mathcal{A}}(a_{1}, \ldots , a_{n})]\subseteq \sigma_{\mathcal{B}}(f(a_{1}), \ldots , f(a_{n}))$.
\end{enumerate}
\end{defi}

We consider two main generalizations of the concept of a logical matrix. ~(1)~ The first, due to Piochi (\cite{Piochi,Piochi2}): a {\em restricted logical matrix}, or {\em Rmatrix}, over a signature $\Theta$ is a triple $\mathcal{M}=(\mathcal{A}, D, \mathcal{F})$ with $\mathcal{A}$ a $\Theta$-algebra; $D$ a subset of the universe of $\mathcal{A}$; and $\mathcal{F}$ a set of homomorphisms $\nu:{\bF}(\theta, \matV)\rightarrow\mathcal{A}$. Given a set of formulas $\Gamma\cup\{\varphi\}$ over $\Theta$, we say $\Gamma$ proves $\varphi$, according to the Rmatrix $\mathcal{M}$, and write $\Gamma\vDash^{\textsf{R}}_{\mathcal{M}}\varphi$ if, for every $\nu\in\mathcal{F}$, $\nu[\Gamma]\subseteq D$ implies $\nu(\varphi)\in D$.~(2)~ The second, due to several authors such as Rescher and Ivlev (\cite{res:62, ivl:73,ivl:85,ivl:88,ivl:13}) and, more recently Avron and Lev (\cite{avr:lev:01, avr:lev:05}): given a signature $\Theta$, a pair $\mathcal{M}=(\mathcal{A}, D)$ is a {\em non-deterministic matrix}, or {\em Nmatrix}, if $\mathcal{A}$ is a $\Theta$-multialgebra and $D$ is a subset of its universe; an Nmatrix defines a consequence operator on the formulas over $\Theta$ for which $\Gamma\vDash_{\mathcal{M}}\varphi$ iff $\nu[\Gamma]\subseteq D$ implies $\nu(\varphi)\in D$ for every homomorphism (of multialgebras) $\nu:{\bF}(\Theta,\matV)\rightarrow\mathcal{A}$.

Although versatile, all these methods have restrictions to their applications: in 1932, G\"odel proved that intuitionistic logic is not characterizable by a single finite logical matrix (\cite{god:32}). Dugundji adapted this proof to show an equivalent result for the modal systems between {\bf S1} and {\bf S5} (\cite{dug:40}). Nmatrices were first considered by Avron and Lev (\cite{avr:lev:01, avr:lev:05}) to deal with paraconsistent logics, especially with \lfis, exactly to overcome uncharacterizability by finite matrices (see, for instance, \cite{avr:05b,avr:07,CCM,CC16}). However, systems such as da Costa's $C_1$, despite being decidable, can not be characterized even by a single finite Nmatrix (\cite{avr:07}). In order to offer finite semantics of non-deterministic character for da Costa's hierarchy and other systems of similar difficulty, we have defined in \cite{ConTol} {\em restricted non-deterministic matrices}, alternatively called {\em restricted Nmatrices} or {\em RNmatrices}, independently defined by \cite{Pawlowski, paw:urb:18}. 

Given a signature $\Theta$, an RNmatrix is a triple $\mathcal{M}=(\mathcal{A}, D, \mathcal{F})$ such that $\mathcal{A}$ is a $\Theta$-multialgebra; $D$ is a subset of the universe of $\mathcal{A}$; and $\mathcal{F}$ is a set of homomorphisms (of multialgebras) $\nu:{\bF}(\Theta, \matV)\rightarrow\mathcal{A}$. As before, we may define a consequence operator as expected: for a set of formulas $\Gamma\cup\{\varphi\}$ over $\Theta$, $\Gamma\vDash^{\textsf{RN}}_{\mathcal{M}}\varphi$ iff, for every $\nu\in\mathcal{F}$, $\nu[\Gamma]\subseteq D$ implies $\nu(\varphi)\in D$. For most of what is to come, structural RNmatrices will be far more relevant: an RNmatrix $\mathcal{M}=(\mathcal{A}, D, \mathcal{F})$ is {\em structural} if, for every substitution $\rho$, $\Gamma\vDash^{\textsf{RN}}_{\mathcal{M}}\varphi$ implies $\rho[\Gamma]\vDash^{\textsf{RN}}_{\mathcal{M}}\rho(\varphi)$. Equivalently, $\mathcal{M}$ is structural if, for every $\nu\in\mathcal{F}$ and substitution $\rho$, $\nu\circ\rho\in\mathcal{F}$.

\subsection{da Costa's Calculi $C_n$, and other \lfis} \label{Cn}
 
We now formally define da Costa's hierarchy for completeness sake. We shall use the signature $\Sigma$ with $\Sigma_{1}=\{\neg\}$, $\Sigma_{2}=\{\vee, \wedge, \rightarrow\}$ and no other connectives. Some abbreviations are then useful to express otherwise excessively long formulas over this signature: for a formula $\alpha$ of $F(\Sigma, \matV)$, $\alpha^0 := \alpha$ and $\alpha^{n+1}:=\neg(\alpha^n \land \neg(\alpha^n))$ $n\in\mathbb{N}$; and $\alpha^{(0)} := \alpha$, $\alpha^{(1)}:=\alpha^1$ and $\alpha^{(n+1)}:=\alpha^{(n)} \land \alpha^{n+1}$, again for $n\in\mathbb{N}$. Inspired by \lfis, we may also denote $\alpha^1=\neg(\alpha \land \neg\alpha)$ by $\alpha^\circ$ (and so $\alpha^{\circ\cdots\circ}$ may designate $\alpha^k$, for $\circ \cdots \circ$ a sequence of $k$ iterations of $\circ$.)

\begin{defi} [The calculi $C_n$, for $n \geq 1$] \label{hilCn} For $n \geq 1$, we define the logic $C_n$ over $\Sigma$ by the following axiom schemata and rules of inference:\\[2mm]
	{\bf Axiom schemata:}\vspace*{-5mm}
	\begin{gather}
	\alpha \imp \big(\beta \imp \alpha\big)             \tag{\kax} \\
	\Big(\alpha\imp\big(\beta\imp\gamma\big)\Big) \imp
	\Big(\big(\alpha\imp\beta\big)\imp\big(\alpha\imp\gamma\big)\Big)
	\tag{\axTrans}\\
	\alpha \imp \Big(\beta \imp \big(\alpha \land \beta\big)
	\Big)  \tag{\axed}\\
	\big(\alpha \land \beta\big) \imp \alpha         \tag{\axeea}\\	
	\big(\alpha \land \beta\big) \imp \beta          \tag{\axeeb}\\
	\alpha \imp \big(\alpha \lor \beta\big)          \tag{\axouda}\\
	\beta \imp \big(\alpha \lor \beta\big)           \tag{\axoudb}\\
	\Big(\alpha \imp \gamma\Big) \imp \Big(
	(\beta \imp \gamma) \imp
	\big(
	(\alpha \lor \beta) \imp \gamma
	\big)\Big)                               \tag{\axoue}\\
	\big(\alpha \imp \beta\big) \lor \alpha          \tag{\axouimp}\\
	\alpha \lor \lnot \alpha                        \tag{\axtnd}\\
	\neg\neg \alpha \imp \alpha
	\tag{\axcf}\\
	\alpha^{(n)} \imp \Big(\alpha \imp \big(\lnot \alpha \imp \beta\big)\Big)
	\tag{\axexp$_n$}\\
	(\alpha^{(n)} \land \beta^{(n)}) \imp \big((\alpha \land \beta)^{(n)} \land (\alpha \lor \beta)^{(n)} \land (\alpha \to \beta)^{(n)}\big)
	\tag{\axp$_n$}
	\end{gather}	
	{\bf Inference rule:}
	\[\frac{\alpha \ \ \ \ \alpha\imp
		\beta}{\beta}  \tag{\MP}\]
\end{defi}

\begin{remark}
Originally (\cite{dC63}) da Costa had considered, instead of (\axexp$_n$), the axiom schema $\alpha^{(n)} \imp \big((\beta \to \alpha) \to ((\beta \to \neg\alpha) \to \neg\beta)\big)$, known as (\axdc$_n$), both easily proven to be equivalent given the other axiom schemata.
\end{remark}


\subsection{RNmatrices for $C_{n}$} \label{sect-Cn}

In \cite{ConTol}, we have constructed RNmatrices $\mathcal{RM}_{C_{n}}=(\mathcal{A}_{C_{n}}, D_{n}, \mathcal{F}_{C_{n}})$ for the calculi $C_{n}$, trough use of swap structures, to achieve rather efficient decision methods for these logics. To give a brief summary of how this was achieved, consider the $(n+1)$-tuples $z=(z_{[1]},z_{[2]},\ldots,z_{[n+1]})$ on $\{0,1\}^{n+1}$ such that $z_{[1]}$ trough $z_{[n+1]}$ are given, respectively, by $\mathsf{b}(\alpha)$, $\mathsf{b}(\neg\alpha)$, $\mathsf{b}(\alpha^1)$, $\mathsf{b}(\alpha^2)$, \ldots, $\mathsf{b}(\alpha^{n-1})$, for a formula $\alpha$ over $\Sigma$ and a $C_n$-bivaluation $\mathsf{b}$ (\cite{lop:alv:80}).\footnote{From now on, the $i$th-coordinate of an $(n+1)$-tuple $z$ on $\{0,1\}^{n+1}$ will be denoted by $z_{[i]}$.} From the properties of a bivaluation, we find that there are precisely $n+2$ of these tuples, which we will call snapshots, namely: $T_n=(1,0,1,\ldots,1)$, $t^n_0=(1,1,0,1,\ldots,1)$, \ldots  , $t^n_{n-2}=(1,1,\ldots,1,0)$, $t^n_{n-1}=(1,1,\ldots,1)$ and $F_n=(0,1,1,\ldots,1)$. It is clear that an $(n+1)$-tuple on $\{0,1\}^{n+1}$ is a snapshot iff it contains at most one $0$, or alternatively, the set of snapshots may be given as 
$$B_n=\{z \in {\bf 2}^{n+1} \ : \ \big(\bigwedge_{i=1}^k z_{[i]}\big) \lor z_{[k+1]} = 1 \ \mbox{ for every } \ 1 \leq k \leq n\}.$$ 

Important subsets of $B_{n}$ are $D_n=\{z \in B_n \ : \ z_{[1]}=1\}$, the set of designated values, and $Boo_n= \{z \in B_n \ : \ z_{[1]} \land z_{[2]}=0\}$, the set of Boolean values, equal respectively to $\{T_{n}, t^{n}_{0}, \ldots  , t^{n}_{n-1}\}$ and $\{T_{n}, F_{n}\}$. Also important are the inconsistent values, $I_{n}=B_{n}\setminus Boo_n$. Notice that $z \in Boo_n$ iff $z=(a,\sneg a, 1, \ldots, 1)$, for an $a \in \{0,1\}$ ($\sneg$ is the Boolean complement in the two-valued Boolean algebra). Now we define the $\Sigma$-multialgebra $\mathcal{A}_{C_{n}}$, with universe $B_n$, as a swap structure (\cite[Chapter~6]{CC16}); for a connective $\sigma$ in $\Sigma$ we will denote its corresponding operation in $\mathcal{A}_{C_{n}}$ as $\tilde{\sigma}$, and for elements $z, w\in B_{n}$, and $\#\in\{\vee, \wedge, \rightarrow\}$. they are given by

\

$\begin{array}{lccl}
(C^n_{\tilde{\neg}}) & \tilde{\neg}\, z &=& \{w \in B_n  \ : \ w_{[1]} = z_{[2]} \ \mbox{ and } \ w_{[2]} \leq z_{[1]}\}\\[4mm]
(C^n_{\tilde{\#}}) & z \,\tilde{\#}\, w &=& \left \{ \begin{tabular}{ll}
$\{u \in Boo_n  \ : \ u_{[1]} = z_{[1]} \# w_{[1]}\}$ & if $z,w\in Boo_n$,\\[3mm]
$\{u \in B_n  \ : \ u_{[1]} = z_{[1]} \# w_{[1]}\}$ & otherwise.\\
\end{tabular}\right.  \\[8mm]
\end{array}$

\

These multioperations of $\A_{C_n}$ may be presented in  a compact form as follows:

\begin{center}
	\begin{tabular}{| c | c | }
		\hline $z$ & $\tilde{\neg}\, z$  \\
		\hline
		
		$T_n$ & $F_n$\\
		\hline
		
		$t^n_i$ &  $D_n$\\
		\hline
		
		$F_n$ &  $T_n$\\
		\hline
		
	\end{tabular}
	\hspace{3.2cm}
	\begin{tabular}{| c | c | c | c |}
		\hline $\tilde{\to}$ & $T_n$ & $t^n_j$ & $F_n$ \\
		\hline
		
		$T_n$ & $T_n$ &  $D_n$ &  $F_n$\\
		\hline
		
		$t^n_i$ &  $D_n$  & $D_n$  & $F_n$\\
		\hline
		
		$F_n$ &  $T_n$ & $D_n$  & $T_n$\\
		\hline
		
	\end{tabular}
\end{center}

\begin{center}
	\begin{tabular}{| c | c | c | c |}
		\hline $\tilde{\land}$ & $T_n$ & $t^n_j$ & $F_n$ \\
		\hline
		
		$T_n$ & $T_n$ &  $D_n$  &  $F_n$\\
		\hline
		
		$t^n_i$ &  $D_n$  & $D_n$  & $F_n$\\
		\hline		
		
		$F_n$ &  $F_n$  & $F_n$  & $F_n$\\
		\hline
		
	\end{tabular}
	\hspace{1cm}
	\begin{tabular}{| c | c | c | c |}
		\hline $\tilde{\lor}$ & $T_n$ & $t^n_j$ & $F_n$ \\
		\hline
		
		$T_n$ & $T_n$ &  $D_n$  &  $T_n$\\
		\hline
		
		$t^n_i$ &  $D_n$  & $D_n$  & $D_n$\\
		\hline
		
		$F_n$ &  $T_n$  & $D_n$  & $F_n$\\
		\hline
		
	\end{tabular}
\end{center}

We finally define the set of restricted homomorphisms $\mathcal{F}_{C_{n}}$, thus finishing the definition of $\mathcal{RM}_{C_{n}}=(\mathcal{A}_{C_{n}}, D_{n}, \mathcal{F}_{C_{n}})$, as the set of all homomorphisms $\nu:{\bF}(\Sigma, \matV)\rightarrow\A_{C_n}$ satisfying that, for every $\alpha$:\\
	
	$\begin{array}{cl}
	(1) & \nu(\alpha) = t^n_0 \ \mbox{ implies that } \ \nu(\alpha \land \neg\alpha)=T_n;\\[1mm]
	(2) & \nu(\alpha) = t^n_{k-1} \ \mbox{ implies that } \ \nu(\alpha \land \neg\alpha) \in I_n$ and $\nu(\alpha^1) = t^n_{k-2},   \\[1mm]
	& \mbox{for every } \ 2\leq k\leq n.\\[3mm]
	\end{array}$

We prove, already in \cite{ConTol}, that $\mathcal{RM}_{C_{n}}$ semantically characterizes $C_{n}$ and, furthermore, that its respective row-branching truth-table is a decision method for this logic.

\section{$\mathcal{B}$-valuations} \label{Bval}

In~\cite[Chapter~6]{CC16} it was shown that, in the case of \lfis\ which are characterized by a single finite Nmatrix such as \mbc, it is possible to replace the underlying two-element Boolean algebra $\B_2$ with domain ${\bf 2}=\{0,1\}$ by an arbitrary (non-trivial)  Boolean algebra  \B. This produces a class of Nmatrices  parametrized by  Boolean algebras, called {\em swap structures semantics}.\footnote{Moreover,  there is a functor from the category of Boolean algebras to the category of swap structures --a full subcategory of the category of multialgebras over $\Sigma^\cons$, the signature obtained from $\Sigma$ by addition of the unary $\cons$ (see~\cite{Coniglio}).} The aim of this generalization is to produce a wider class of models in order to study these logics by adapting the tools from algebraic logics to the context of multialgebras (see, for instance, \cite{Coniglio}). However, logics such as $C_{1}$ lie outside the scope of swap structures semantics (this is related to the uncharacterizability of this logic by a single finite Nmatrix, as mentioned above). The aim of the next two sections is, in the same way as the class of swap structures generalize finite Nmatrices defined over $\B_2$ to any Boolean algebra \B, to generalize the RNmatrix $\matGM_{C_n}$ to any Boolean algebra \B.

In order to do this observe that  it is possible to replace, in the definitions from the previous section, the Boolean algebra  $\B_2$ by an arbitrary Boolean algebra  \B.  Notice first that all the notions concerning the Nmatrix $\matM_{C_n} = (\mathcal{A}_{C_{n}}, D_{n})$ underlying the RNmatrix $\matGM_{C_n}$ were presented in general terms, involving the Boolean operators of  $\B_2$ and the elements $0$ and $1$ of {\bf 2} (which are present in any \B). To begin with, the domain $B_n$, the sets $D_n$ and $Boo_n$, as well as the multioperations of $\A_{C_n}$, can be easily defined over any Boolean algebra \B. In the case of bivaluations, some small adjustments are required in order to generalize to arbitrary Boolean algebras. From now on, only  non-trivial Boolean algebras will be considered.\footnote{A Boolean algebra is non-trivial if $0\neq 1$, which is equivalent to say that it has at least two elements. As it was done with $\B_2$, the Boolean operation corresponding to each binary connective $\#$ of $\Sigma$ will be also written as $\#$. The  Boolean complement in a Boolean algebra will be denoted by $\sneg$.}

\begin{defi} \label{B-val-def}
	Let \B\ be a Boolean algebra with domain $|\B|$.
	A {\em \B-valuation} for $C_n$ is a function $\mathsf{b}:\bF(\Sigma,\matV) \to |\B|$ satisfying the following clauses:\\[2mm]
	$\begin{array}{ll}
	(V1) & \mathsf{b}(\alpha \# \beta)=\mathsf{b}(\alpha)\,\#\, \mathsf{b}(\beta) \ \mbox{ for } \# \in \{\land,\lor,\to\}; \\[1mm]
	(V2) & \sneg\mathsf{b}(\alpha) \leq \mathsf{b}(\neg\alpha);\\[1mm]
	(V3) & \mathsf{b}(\neg\neg\alpha) \leq  \mathsf{b}(\alpha);\\[1mm]
		(V4)_n & \mathsf{b}(\alpha^n)= \sneg(\mathsf{b}(\alpha^{n-1}) \land \mathsf{b}(\neg (\alpha^{n-1})));\\[1mm]
	(V5) & \mathsf{b}(\neg(\alpha^\circ))=\mathsf{b}(\alpha) \land \mathsf{b}(\neg \alpha);\\[1mm]
	(V6)_n & \mathsf{b}(\alpha^{(n)})\land \mathsf{b}(\beta^{(n)}) \leq  \mathsf{b}((\alpha\#\beta)^{(n)}) \ \mbox{ for } \# \in \{\land,\lor,\to\}.\\[1mm]
\end{array}$
\end{defi}

\noindent The semantical consequence relation  w.r.t. \B-valuations for $C_n$, in which $1$ is the only designated value, will be denoted by $\vDash_{n}^\B$. Thus, $\Gamma \vDash_{n}^\B \varphi$ iff $\mathsf{b}(\varphi)=1$, for any \B-valuation $\mathsf{b}$ for $C_n$ such that $\mathsf{b}(\gamma)=1$ for every $\gamma \in \Gamma$. The semantical consequence with respect to \B-valuations for every Boolean algebra \B\ will be denoted by  $\vDash_{n}$.
Then, $\Gamma \vDash_{n} \varphi$ iff $\Gamma \vDash_{n}^\B \varphi$, for every \B.

\begin{remarks} \label{rem-B-val}
Let $\mathsf{b}$ be a \B-valuation for $C_n$, and let $\alpha$ be a formula. Let $z=(z_{[1]},z_{[2]},\ldots,z_{[n+1]})$ in $\B^{n+1}$ be such that each coordinate is given by $\mathsf{b}(\alpha)$, $\mathsf{b}(\neg\alpha)$, $\mathsf{b}(\alpha^1)$, $\mathsf{b}(\alpha^2)$, \ldots, $\mathsf{b}(\alpha^{n-1})$, respectively.\\[1mm]
(1) By iterating clause $(V5)$ of Definition~\ref{B-val-def}, we obtain the following:
 $$\begin{array}{lllll}
\mathsf{b}(\neg(\alpha^1))&= & \mathsf{b}(\alpha) \land \mathsf{b}(\neg\alpha),\\[2mm]
\mathsf{b}(\neg(\alpha^2))&=& \mathsf{b}(\alpha^1) \land \mathsf{b}(\neg(\alpha^1))&=& \mathsf{b}(\alpha) \land \mathsf{b}(\neg\alpha) \land  \mathsf{b}(\alpha^1),\\[2mm]
\vdots\\[2mm]
\mathsf{b}(\neg(\alpha^{n-1}))&=& \mathsf{b}(\alpha^{n-2}) \land \mathsf{b}(\neg(\alpha^{n-2}))&=& \mathsf{b}(\alpha) \land \mathsf{b}(\neg\alpha) \land  \bigwedge_{i=1}^{n-2}\mathsf{b}(\alpha^{i}).
\end{array}$$
From this it follows that at most one of the coordinates of $z$ can be $0$. Moreover, if $z_{[i]}=0$ then $z_{[k]}=1$ for every $i+1 \leq k \leq n+1$.\\[1mm]
(2) Using clause $(V4)_n$ and  item~(1) it follows that
$$\mathsf{b}(\alpha^n)= \sneg(\mathsf{b}(\alpha^{n-1}) \land \mathsf{b}(\neg (\alpha^{n-1}))) = $$
$$\sneg(\mathsf{b}(\alpha) \land \mathsf{b}(\neg\alpha) \land  \mathsf{b}(\alpha^1) \land \ldots\land  \mathsf{b}(\alpha^{n-2}) \land  \mathsf{b}(\alpha^{n-1})).$$
That is,
$$\sneg\mathsf{b}(\alpha^n) = \mathsf{b}(\alpha) \land \mathsf{b}(\neg\alpha) \land  \mathsf{b}(\alpha^1) \land \ldots\land  \mathsf{b}(\alpha^{n-2}) \land  \mathsf{b}(\alpha^{n-1}).$$
One also sees that $\mathsf{b}(\alpha^{(1)})=\mathsf{b}(\alpha^{1})=z_{[3]}$, $\mathsf{b}(\alpha^{(2)})=\mathsf{b}(\alpha^{(1)})\land\mathsf{b}(\alpha^{2})=\mathsf{b}(\alpha^{1})\land\mathsf{b}(\alpha^{2})=z_{[3]}\land z_{[4]}$ and, inductively, $\mathsf{b}(\alpha^{(n-1)})=\bigwedge_{i=1}^{n-1}\mathsf{b}(\alpha^{i})=\bigwedge_{i=3}^{n+1}z_{[i]}$.\\[1mm]
(3) Let $a=\bigwedge_{i=1}^{n-1} \mathsf{b}(\alpha^i)=\bigwedge_{i=3}^{n+1} z_{[i]}$. From clause~(V1) and items~(1) and~(2) we obtain the following:
$$\mathsf{b}(\alpha^{(n)})=\bigwedge_{i=1}^{n} \mathsf{b}(\alpha^i)= a \land \sneg ((\mathsf{b}(\alpha) \land \mathsf{b}(\neg\alpha)) \land  a)  = \sneg(\mathsf{b}(\alpha) \land \mathsf{b}(\neg\alpha)) \land a.$$
(4) Finally, by clause $(V2)$ and item~(1) we obtain that:\\[2mm]
\indent $\begin{array}{l}
\mathsf{b}(\alpha) \lor \mathsf{b}(\neg\alpha)=1,\\[2mm]
(\mathsf{b}(\alpha) \land \mathsf{b}(\neg\alpha)) \lor \mathsf{b}(\alpha^1) =1,\\[2mm]
\vdots\\[2mm]
(\mathsf{b}(\alpha) \land \mathsf{b}(\neg\alpha) \land \mathsf{b}(\alpha^1) \land \ldots \land \mathsf{b}(\alpha^{n-2})) \lor \mathsf{b}(\alpha^{n-1}) =1.\\[1mm]
\end{array}$
\end{remarks}

\begin{prop} \label{B8-bival}
Let $\mathsf{b}$ be a \B-valuation for $C_n$, and let $\alpha$ be a formula. Then: $\mathsf{b}(\neg (\alpha\#\beta))=\sneg\mathsf{b}(\alpha\#\beta)$ whenever $\mathsf{b}(\neg \alpha)=\sneg\mathsf{b}(\alpha)$ and  $\mathsf{b}(\neg \beta)=\sneg\mathsf{b}(\beta)$,
for $\# \in \{\land,\lor,\to\}$. Also, $\mathsf{b}(\neg\neg \alpha)=\sneg\mathsf{b}(\neg\alpha)$, provided that $\mathsf{b}(\neg \alpha)=\sneg\mathsf{b}(\alpha)$.
\end{prop}
\begin{proof}
Let $\mathsf{b}$ be a  \B-valuation such that $\mathsf{b}(\neg \alpha)=\sneg\mathsf{b}(\alpha)$ and  $\mathsf{b}(\neg \beta)=\sneg\mathsf{b}(\beta)$. Then $\mathsf{b}(\alpha) \land\mathsf{b}(\neg \alpha)=\mathsf{b}(\beta) \land\mathsf{b}(\neg \beta)=0$ and so $\sneg(\mathsf{b}(\alpha) \land\mathsf{b}(\neg \alpha))=\sneg(\mathsf{b}(\beta) \land\mathsf{b}(\neg \beta))=1$. By Remark~\ref{rem-B-val}(1), $\mathsf{b}(\neg(\alpha^i))=\mathsf{b}(\neg(\beta^i))=0$ and so, by~$(V2)$, $\mathsf{b}(\alpha^i)=\mathsf{b}(\beta^i)=1$ for $1 \leq i \leq n-1$. By Remark~\ref{rem-B-val}(3), $\mathsf{b}(\alpha^{(n)})=\mathsf{b}(\beta^{(n)})=1$ and so, by~$(V6)_n$, $\mathsf{b}((\alpha\#\beta)^{(n)})=1$. But $\mathsf{b}((\alpha\#\beta)^{(n)}) \leq \sneg(\mathsf{b}(\alpha\#\beta) \land \mathsf{b}(\neg(\alpha\#\beta)))$, by Remark~\ref{rem-B-val}(3), hence $\mathsf{b}(\alpha\#\beta) \land \mathsf{b}(\neg(\alpha\#\beta))=0$. That is, $\mathsf{b}(\neg (\alpha\#\beta))=\sneg\mathsf{b}(\alpha\#\beta)$, by~$(V2)$.

Finally, suppose that  $\mathsf{b}(\neg \alpha)=\sneg\mathsf{b}(\alpha)$. By $(V3)$, $\mathsf{b}(\neg\neg \alpha) \land \mathsf{b}(\neg\alpha) \leq \mathsf{b}(\alpha) \land \mathsf{b}(\neg\alpha)=0$. Hence  $\mathsf{b}(\neg\neg \alpha)=\sneg\mathsf{b}(\neg\alpha)$, given that $\mathsf{b}(\neg\neg \alpha) \lor \mathsf{b}(\neg\alpha)=1$ by $(V2)$.
\end{proof}

\begin{coro}
$\B_2$-valuations for $C_n$ coincide with bivaluations for $C_n$.
\end{coro}
\begin{proof}
The only clause that deserves some attention is~(B8), given that the proof of the validity of the other clauses is immediate. Thus, let $\mathsf{b}$ be a  $\B_2$-valuation such that $\mathsf{b}(\alpha)\neq \mathsf{b}(\neg \alpha)$ and  $\mathsf{b}(\beta)\neq \mathsf{b}(\neg \beta)$. Then, either $\mathsf{b}(\alpha)=0$ or $\mathsf{b}(\neg\alpha)=0$, and the same holds for $\beta$. From this, $\mathsf{b}(\alpha)\land \mathsf{b}(\neg \alpha)=0$ and  $\mathsf{b}(\beta)\land \mathsf{b}(\neg \beta)=0$. This means that $\mathsf{b}(\neg\alpha) = \sneg \mathsf{b}(\alpha)$ and  $\mathsf{b}(\neg\beta) = \sneg\mathsf{b}(\beta)$, by~$(V2)$.  Thus, $\mathsf{b}(\neg (\alpha\#\beta))=\sneg\mathsf{b}(\alpha\#\beta)$, for $\# \in \{\land,\lor,\to\}$, because of Proposition~\ref{B8-bival}. But this is equivalent to $\mathsf{b}(\alpha\#\beta) \neq \mathsf{b}(\neg (\alpha\#\beta))$, for $\# \in \{\land,\lor,\to\}$.
\end{proof}

\

\noindent The proposition above shows that $\B$-valuations generalize bivaluations to arbitrary (non-trivial) Boolean algebras. The following result follows easily from the completeness of $C_n$ with respect to bivaluations:

\begin{theorem} [Soundness and completeness of $C_n$ w.r.t. \B-valuations] \label{comple-Cn-B-val} \ \\
Let \B\ be a Boolean algebra. Fix $n \geq 1$, and let $\Gamma \cup \{\varphi\} \subseteq \bF(\Sigma,\matV)$. Then: $\Gamma \vdash_{C_n} \varphi$ \ iff \ $\Gamma \vDash_{n} \varphi$. 
\end{theorem}
\begin{proof}
`Only if' part (Soundness). It is clear that, if $\varphi$ is  an  instance of an axiom of $C_n$ then $\mathsf{b}(\varphi)=1$, for any \B-valuation $\mathsf{b}$ for $C_n$. On the other hand, by $(V1)$ and the properties of the Boolean implication $a \to b:=\sneg a \lor b$ it follows that $\mathsf{b}(\varphi)=\mathsf{b}(\varphi \to \psi)=\mathsf{b}(\varphi) \to \mathsf{b}(\psi)=1$ implies $\mathsf{b}(\psi)=1$. From this, by induction on the length of a derivation of $\varphi$ from $\Gamma$ in $C_n$ it follows that $\Gamma \vdash_{C_n} \varphi$ implies that  $\Gamma \vDash_{n}^\B \varphi$, for every \B. Therefore, $\Gamma \vDash_{n} \varphi$. \\[2mm]
`If' part (Completeness). Assume that $\Gamma \vDash_{n}^\B \varphi$, for every \B. In particular, $\Gamma \vDash_{n}^{\B_2} \varphi$ and so, by completeness of $C_n$  w.r.t. bivaluations, it follows that $\Gamma \vdash_{C_n} \varphi$.	
\end{proof}

\
\section{Restricted swap structures for $C_{n}$} \label{RS}

In this section, the RNmatrix $\matGM_{C_n} = ( \A_{C_n}, D_n,\matF_{C_n} )$ introduced in Section~\ref{sect-Cn} will be generalized to what we call a {\em restricted swap structures semantics} for $C_n$. To be more precise, the Nmatrix $\matM_{C_n} = ( \A_{C_n}, D_n )$ will be extended, for any \B, to an Nmatrix $\matM_{C_n}^\B = ( \A_{C_n}^\B, D_n^\B )$. The latter is the Nmatrix associated to the swap structure  $\A_{C_n}^\B$ over the Boolean algebra \B, to be defined by extending $\A_{C_n}$ to any \B. After this, the set of valuations  $\matF_{C_n}$ over $\matM_{C_n}$  will be accordingly extended to a set of valuations  $\matF_{C_n}^\B$ over $\matM_{C_n}^\B$, obtaining so a restricted swap structures semantics  $\matGS_{C_n}$ formed by the class of all the RNmatrices  for $C_n$ of the form $\matGM_{C_n}^\B= ( \A_{C_n}^\B, D_n^\B,\matF_{C_n}^\B )$.

As it was done in the previous sections with the structures over ${\bf 2}^{n+1}$, if $z \in \B^{n+1}$ then its $i$th-coordinate will be denoted by $z_{[i]}$.
As in Remarks~\ref{rem-B-val}, consider $(n+1)$-uples $z=(z_{[1]},z_{[2]},\ldots,z_{[n+1]})$ in $\B^{n+1}$ such that each coordinate is given by $\mathsf{b}(\alpha)$, $\mathsf{b}(\neg\alpha)$, $\mathsf{b}(\alpha^1)$, $\mathsf{b}(\alpha^2)$, \ldots, $\mathsf{b}(\alpha^{n-1})$, respectively, for a given  \B-valuation $\mathsf{b}$ for $C_n$ and a given formula $\alpha$. The idea is, from  the properties of \B-valuations analyzed in Remarks~\ref{rem-B-val}, abstracting $\mathsf{b}$ and considering just the $(n+1)$-uples with multioperations between them, reflecting such properties. This lead us to the following definition:

\begin{defi} \label{def-sets-B} The  set of {\em \B-snapshots}  for $C_n$ is given by
$$B_n^\B=\{z \in |\B|^{n+1} \ : \ \big(\bigwedge_{i=1}^k z_{[i]}\big) \lor z_{[k+1]} = 1 \ \mbox{ for every } \ 1 \leq k \leq n\}.$$	
Consider the following subsets of $B_n^\B$:
	\begin{itemize}
		\item[-] $D_n^\B:= \{z \in B_n^\B \ : \ z_{[1]}=1\}$ ({\em designated values});
		\item[-] $Boo_n^\B:= \{z \in B_n^\B \ : \ z_{[1]} \land z_{[2]}=0\}$ ({\em Boolean values}).
	\end{itemize}
\end{defi}

\noindent Observe that $B_n^\B$ is defined according to Remarks~\ref{rem-B-val}(4). The set of designated values is defined according to the intended meaning of the coordinates of the snapshots. In addition, by the proof of Proposition~\ref{B8-bival} it follows that, if $\mathsf{b}(\neg \alpha)=\sneg\mathsf{b}(\alpha)$, then $\mathsf{b}(\alpha^i)=1$ for $1 \leq i \leq n-1$. This is reflected by the fact that $z \in Boo^\B_n$ iff $z=(a,\sneg a,1, \ldots,1)$ for some $a \in |\B|$.

The restrictions imposed to the binary multioperators are justified by Proposition~\ref{B8-bival} in combination with  the intended meaning of the coordinates of the snapshots.
The extension of $\A_{C_n}$ to any Boolean algebra \B\ is then defined as follows:

\begin{defi} \label{defACnB} Let \B\ be a Boolean algebra. The {\em (full) swap structure for $C_n$ over \B} is the multialgebra $\A_{C_n}^\B = ( B_n^\B, \tilde{\land}, \tilde{\lor}, \tilde{\to}, \tilde{\neg})$ over $\Sigma$ defined as follows, for any $z,w \in B_n^\B$:\\
	
	$\begin{array}{llcl}
	(C^{\B,n}_{\tilde{\neg}}) & \tilde{\neg}\, z &=& \{w \in B_n^\B  \ : \ w_{[1]} = z_{[2]} \ \mbox{ and } \ w_{[2]} \leq z_{[1]}\}\\[4mm]
	(C^{\B,n}_{\tilde{\#}}) &  z \,\tilde{\#}\, w &=& \left \{ \begin{tabular}{ll}
	$\{u \in Boo_n^\B  \ : \ u_{[1]} = z_{[1]} \# w_{[1]}\}$ & if $z,w\in Boo_n^\B$,\\[3mm]
	$\{u \in B_n^\B  \ : \ u_{[1]} = z_{[1]} \# w_{[1]}\}$ & otherwise,\\
	\end{tabular}\right.  \\[8mm]
	&&&\mbox{ for } \# \in \{\land, \lor, \to\}.
	\end{array}$
\end{defi}

\begin{remark} \label{exte-B}
Since $0,1 \in \B$ for any \B\ then (up to names) we have that $B_n \subseteq B_n^\B$ for every $n \geq 1$. Moreover, $D_n^\B \cap B_n = D_n$ and $Boo_n^\B \cap B_n  = Boo_n = \{T_n,F_n\}$. Furthermore, it its easy to see that the multioperations of $\A_{C_n}^\B$, when restricted to $B_n$, coincide with that of $\A_{C_n}$. This means that $\A_{C_n}$ is a submultialgebra of $\A_{C_n}^\B$, for every $\B$.
\end{remark}

\noindent  If $z \in Boo_n^\B$ then $z=(a,\sneg a,1,\ldots,1)$ and so $\tilde{\neg}\, z = \{(\sneg a,a,1,\ldots,1)\}$. Indeed, by definition of $\tilde{\neg}$ in  $\A^\B_{C_n}$, $\tilde{\neg}\, z = \{ w \in B_n^\B \ : \ w_{[1]}=\sneg a$ and $w_{[2]} \leq a \}$. But then $w_{[1]} \land w_{[2]} \leq \sneg a \land a =0$. Hence, $w_{[2]}=a$, given that $w_{[1]} \lor w_{[2]}=1$ by definition of $B_n^\B$. That is, $\tilde{\neg}\, z = \{(\sneg a,a,1,\ldots,1)\}$.

\begin{defi} \label{swap-Cn-B} The {\em Nmatrix induced by $\A_{C_n}^\B$} is  $\matM_{C_n}^\B :=( \A_{C_n}^\B, D_n^\B )$.
\end{defi}

\begin{defi} \label{valCnB} Let \B\ be a Boolean algebra. The {\em restricted Nmatrix for $C_n$ over \B} is the RNmatrix $\matGM_{C_n}^\B := ( \A_{C_n}^\B, D_n^\B,\matF_{C_n}^\B )$ obtained from the Nmatrix 	$\matM_{C_n}^\B$ by considering the set  $\matF_{C_n}^\B$ of valuations $\nu$ over $\A_{C_n}^\B$ (that is, homomorphisms of multialgebras $\nu:\bF(\Sigma,\matV) \to \A_{C_n}^\B$) such that, for all formulas $\alpha$ and $\beta$ (recalling that $\alpha^1=\neg(\alpha \land \neg\alpha)$):
	
	$\begin{array}{cl}
	(1) & \nu(\alpha \land \neg\alpha) \in \{z \in \nu(\alpha) \,\tilde{\land}\, \nu(\neg\alpha) \ : \  z_{[2]} = \nu(\alpha)_{[3]}\};\\[1mm]
	(2) & \nu(\alpha^1) = (\nu(\alpha)_{[3]}, \nu(\alpha)_{[1]} \land \nu(\alpha)_{[2]},\nu(\alpha)_{[4]},\ldots,\nu(\alpha)_{[n+1]},\sneg\big(\bigwedge_{i=1}^{n+1} \nu(\alpha)_{[i]} \big));\\[1mm]
	(3) & \nu((\alpha^{(n)} \land \beta^{(n)}) \to (\alpha\# \beta)^{(n)}) \in D_n^B \ \mbox{ for } \# \in \{\land, \lor, \to\}. \\[3mm]
	\end{array}$
\end{defi}

\noindent
It should be clear that $\matGM_{C_n}^\B$ is structural.
As it will be shown in Proposition~\ref{val-nontriv},  it is possible, in general, to define valuations satisfying all these requirements.

\begin{defi} 
The semantical consequence relation  w.r.t. the RNmatrix $\matGM_{C_n}^\B$ will be denoted by $\vDash_{\matGM_{C_n}^\B}^\mathsf{RN}$.		
The {\em restricted swap structures semantics for $C_n$} is the class $\matGS_{C_n}$  formed by the RNmatrices  $\matGM_{C_n}^\B$ such that \B\ is a Boolean algebra. 
The semantical consequence relation  w.r.t. $\matGS_{C_n}$, denoted by $\vDash_{\matGS_{C_n}}^\mathsf{RN}$, is defined as follows: $\Gamma \vDash_{\matGS_{C_n}}^\mathsf{RN} \varphi$ iff $\Gamma \vDash_{\matGM_{C_n}^\B}^\mathsf{RN} \varphi$, for every Boolean algebra \B.
\end{defi}

\begin{remark} \label{obs-val}
Let $\nu(\alpha)=(z_{[1]},z_{[2]},\ldots,z_{[n+1]})$ for a given $\nu \in \matF_{C_n}^\B$ and $\alpha$. Then, by Definition~\ref{valCnB}:\\[3mm]
\indent $\begin{array}{l}
\nu(\alpha \land \neg\alpha)_{[1]}=z_{[1]} \land z_{[2]} \ \mbox{ and } \ \nu(\alpha \land \neg\alpha)_{[2]}=z_{[3]};\\[2mm]
\nu(\alpha^1) =(z_{[3]},z_{[1]} \land z_{[2]},z_{[4]},z_{[5]},\ldots,z_{[n+1]},\sneg\bigwedge_{i=1}^{n+1}z_{[i]});\\[2mm]
\nu(\alpha^1 \land \neg(\alpha^1))_{[1]}=z_{[1]} \land z_{[2]} \land z_{[3]} \ \mbox{ and } \ \nu(\alpha \land \neg\alpha)_{[2]}=z_{[4]};\\[2mm]
\vdots\\[2mm]
\nu(\alpha^{n-2} \land \neg(\alpha^{n-2}))_{[1]}=\bigwedge_{i=1}^{n}z_{[i]} \ \mbox{ and } \ \nu(\alpha^{n-2} \land \neg(\alpha^{n-2}))_{[2]}=z_{[n+1]};\\[2mm]
\nu(\alpha^{n-1}) =(z_{[n+1]},\bigwedge_{i=1}^{n}z_{[i]},\sneg\bigwedge_{i=1}^{n+1}z_{[i]},1,1, \ldots,1);\\[2mm]
\nu(\alpha^{n-1} \land \neg(\alpha^{n-1}))_{[1]}=\bigwedge_{i=1}^{n+1}z_{[i]} \ \mbox{ and } \\[2mm]
\multicolumn{1}{r}{\nu(\alpha^{n-1} \land \neg(\alpha^{n-1}))_{[2]}=\sneg\bigwedge_{i=1}^{n+1}z_{[i]};}\\[2mm]
\nu(\alpha^{n}) =(\sneg\bigwedge_{i=1}^{n+1}z_{[i]},\bigwedge_{i=1}^{n+1}z_{[i]},1,1, \ldots,1);\\[2mm]
\nu(\alpha^{(n)})_{[1]} = \bigwedge_{i=3}^{n+1}z_{[i]} \land \sneg\bigwedge_{i=1}^{n+1}z_{[i]} = \big(\bigwedge_{i=3}^{n+1}z_{[i]}\big) \land\sneg(z_{[1]} \land z_{[2]}).\\[2mm]
\end{array}$
\end{remark}

\noindent 
Recall from Subsection~\ref{sect-Cn} the strong negation $\sneg \alpha: = \neg \alpha \land  \alpha^{(n)}$ definable in $C_n$. The previous Remark~\ref{obs-val} allows to show the following:

\begin{prop}\label{strong_negation}
 For every $\nu \in \matF_{C_n}^\B$ and every formula $\alpha$:\\[1mm]
(1) $\nu(\alpha \land \sneg\alpha) = F_n$;\\[1mm]
(2) $\nu(\alpha \lor \sneg\alpha) \in D_n^\B$.
\end{prop}
\begin{proof}
We will adopt the notation $\nu(\alpha)=(z_{[1]},z_{[2]},\ldots,z_{[n+1]})$ of Remark~\ref{obs-val}. In addition, let $a:=\bigwedge_{i=3}^{n+1} z_{[i]}$.\\[1mm]
(1) Observe that $\nu(\neg\alpha)_{[1]}=z_{[2]}$, and $\nu(\alpha^{(n)})_{[1]} = \sneg(z_{[1]} \land z_{[2]}) \land a$, by Remark~\ref{obs-val}. Then, $\nu(\sneg \alpha)_{[1]}=\nu(\neg \alpha \land  \alpha^{(n)})_{[1]} = z_{[2]} \land\sneg(z_{[1]} \land z_{[2]}) \land a= z_{[2]} \land(\sneg z_{[1]} \lor \sneg z_{[2]}) \land a=z_{[2]} \land \sneg z_{[1]} \land a$. From this, it follows that $\nu(\alpha \land \sneg\alpha)_{[1]} = z_{[1]} \land (z_{[2]} \land \sneg z_{[1]} \land a)=0$. That is, $\nu(\alpha \land \sneg\alpha) = F_n$.\\[1mm]
(2) Since $\nu(\alpha) \in B_n^\B$, $z_{[1]} \lor z_{[2]} = 1$. In addition, $(z_{[1]} \land z_{[2]}) \lor z_{[3]} =1$, hence $\sneg z_{[3]} \leq z_{[1]} \land z_{[2]} \leq z_{[1]}$. Analogously, $(z_{[1]} \land z_{[2]} \land z_{[3]}) \lor z_{[4]} =1$, hence $\sneg z_{[4]} \leq z_{[1]} \land z_{[2]} \land z_{[3]} \leq z_{[1]}$. In the same way we can prove that $\sneg z_{[i]} \leq z_{[1]}$ for $3 \leq i \leq n+1$ and so $\sneg a = \sneg \bigwedge_{i=3}^{n+1} z_{[i]} = \bigvee_{i=3}^{n+1} \sneg z_{[i]} \leq z_{[1]}$. Therefore $\nu(\alpha \lor \sneg\alpha)_{[1]} = z_{[1]} \lor (z_{[2]} \land \sneg z_{[1]} \land a) = (z_{[1]} \lor z_{[2]}) \land (z_{[1]} \lor \sneg z_{[1]}) \land (z_{[1]} \lor a) = 1$. That is,  $\nu(\alpha \lor \sneg\alpha) \in D_n^\B$.
\end{proof}

The following can be easily proved by induction.

\begin{lemma} \label{prop-val}
For $2\leq k\leq n-2$, we have that, for a homomorphism $\nu \in\mathcal{F}_{C_{n}}^{\mathcal{B}}$,
\[ \nu(\alpha^{k})=( \nu(\alpha)_{[k+2]}, \bigwedge_{i=1}^{k+1} \nu(\alpha)_{[i]},  \nu(\alpha)_{[k+3]}, \ldots,   \nu(\alpha)_{[n+1]}, \sneg \bigwedge_{i=1}^{n+1} \nu(\alpha)_{[i]}, 1, \ldots  , 1).\]
This way, 
\[\nu(\alpha^{n-1})=( \nu(\alpha)_{[n+1]}, \bigwedge_{i=1}^{n} \nu(\alpha)_{[i]}, \sneg \bigwedge_{i=1}^{n+1} \nu(\alpha)_{[i]}, 1, \ldots  , 1)\]
and $ \nu(\alpha^{n})=(\sneg \bigwedge_{i=1}^{n+1} \nu(\alpha)_{[i]}, \bigwedge_{i=1}^{n+1} \nu(\alpha)_{[i]}, 1, \ldots  , 1)$.
\end{lemma}

\begin{lemma} \label{lem-sound-Cn-B}
	Let $\nu$ be a valuation in $\matF_{C_n}^\B$. Then, the mapping $\mathsf{b}:\bF(\Sigma,\matV) \to |\B|$ given by $\mathsf{b}(\alpha):=\nu(\alpha)_{[1]}$ is a \B-valuation for $C_n$ such that  $\mathsf{b}(\alpha)=1$ iff $\nu(\alpha) \in D_n^\B$  for every formula $\alpha$. 
\end{lemma}

\begin{proof}
Clause $(V1)$ is quite obvious: since $\nu$ is a homomorphism, $\nu(\alpha\#\beta)$ is in $\nu(\alpha)\,\tilde{\#}\, \nu(\beta)$, for any $\#\in\{\lor, \land, \to\}$. Given that $u\in z\,\tilde{\#}\,w$ if, and only if, $u_{[1]}=z_{[1]}\# w_{[1]}$, we obtain that $\nu(\alpha\#\beta)_{[1]}= \nu(\alpha)_{[1]}\# \nu(\beta)_{[1]}$ and therefore $\mathsf{b}(\alpha\#\beta)=\mathsf{b}(\alpha)\#\mathsf{b}(\beta)$. It holds that $\mathsf{b}(\neg\alpha)= \nu(\neg\alpha)_{[1]}$, and $\nu(\neg\alpha)\in \tilde{\neg} \, \nu(\alpha)$, since $\nu$ is a homomorphism. But $w\in \tilde{\neg}\,z$ if, and only if $w_{[1]}=z_{[2]}$ and $w_{[2]}\leq z_{[1]}$, hence $ \nu(\neg\alpha)_{[1]}= \nu(\alpha)_{[2]}$. From the definition of $B_{n}^{\mathcal{B}}$, $\nu(\alpha)_{[1]}\lor\nu(\alpha)_{[2]}=1$, implying that $\nu(\alpha)_{[2]}\geq \sneg  \nu(\alpha)_{[1]}$. This means that $\mathsf{b}(\neg\alpha)\geq{\sim}\mathsf{b}(\alpha)$, which corresponds to clause $(V2)$. Again from the definition of $\tilde{\neg}$, $\nu(\neg\neg\alpha)_{[1]}= \nu(\neg\alpha)_{[2]}$, and $\nu(\neg\alpha)_{[2]}\leq  \nu(\alpha)_{[1]}$, meaning that $\mathsf{b}(\neg\neg\alpha)= \nu(\neg\neg\alpha)_{[1]}\leq  \nu(\alpha)_{[1]}=\mathsf{b}(\alpha)$. This corresponds to clause $(V3)$.

From Lemma~\ref{prop-val}, we have that $\nu(\alpha^{n})_{[1]}=\sneg \bigwedge_{i=1}^{n+1} \nu(\alpha)_{[i]}$, $\nu(\alpha^{n-1})_{[1]}= \nu(\alpha)_{[n+1]}$ and $\nu(\neg(\alpha^{n-1}))_{[1]}= \nu(\alpha^{n-1})_{[2]}=\bigwedge_{i=1}^{n} \nu(\alpha)_{[i]}$, meaning therefore that 
\[\mathsf{b}(\alpha^{n})=\sneg \bigwedge_{i=1}^{n+1} \nu(\alpha)_{[i]}=\sneg (\nu(\alpha)_{[n+1]}\land \bigwedge_{i=1}^{n} \nu(\alpha)_{[i]})=\sneg (\mathsf{b}(\alpha^{n-1})\land\mathsf{b}(\neg(\alpha^{n-1}))),\]
that is, clause $(V4)_{n}$ is satisfied. We have that $\mathsf{b}(\neg(\alpha^{\circ}))=\nu(\neg(\alpha^{1}))_{[1]}= \nu(\alpha^{1})_{[2]}= \nu(\alpha)_{[1]}\land\nu(\alpha)_{[2]}= \nu(\alpha)_{[1]}\land\nu(\neg\alpha)_{[1]}=\mathsf{b}(\alpha)\land\mathsf{b}(\neg\alpha)$, what validates clause $(V5)$. From the definition of $\mathcal{F}_{C_{n}}^{\mathcal{B}}$  it holds that, for any $\#\in\{\lor, \land, \to\}$, $\nu((\alpha^{(n)}\land\beta^{(n)})\to(\alpha\#\beta)^{(n)})\in D_{n}^{\mathcal{B}}$, meaning that $(\nu(\alpha^{(n)})_{[1]}\land\nu(\beta^{(n)})_{[1]})\to\nu((\alpha\#\beta)^{(n)})_{[1]}= \nu((\alpha^{(n)}\land\beta^{(n)})\to(\alpha\#\beta)^{(n)})_{[1]}=1$. This implies that $\mathsf{b}(\alpha^{(n)})\land\mathsf{b}(\beta^{(n)})= \nu(\alpha^{(n)})_{[1]}\land\nu(\beta^{(n)})_{[1]}\leq  \nu((\alpha\#\beta)^{(n)})_{[1]}=\mathsf{b}((\alpha\#\beta)^{(n)})$, showing that condition $(V6)_{n}$ is also validated.

Clearly, $\mathsf{b}(\alpha)=1$ if, and only if, $\nu(\alpha)_{[1]}=1$, which is equivalent to $\nu(\alpha)\in D_{n}^{\mathcal{B}}$.
\end{proof}

\begin{lemma} \label{lem-compl-Cn-B}
	For any \B-valuation  $\mathsf{b}$ for $C_n$, the mapping $\nu:\bF(\Sigma,\matV) \to B_n^\B$ given by $\nu(\alpha):=(\mathsf{b}(\alpha),\mathsf{b}(\neg\alpha),\mathsf{b}(\alpha^1),\mathsf{b}(\alpha^2),\ldots,\mathsf{b}(\alpha^{n-1}))$  is  a valuation in $\matF_{C_n}^\B$ such that  $\mathsf{b}(\alpha)=1$ iff $\nu(\alpha) \in D_n^\B$ for every formula $\alpha$.
\end{lemma}

\begin{proof}
First of all, we prove that $\nu$ is a homomorphism. By definition of $\nu$, $\nu(\neg\alpha)_{[1]}=\mathsf{b}(\neg\alpha)= \nu(\alpha)_{[2]}$ and, by $(V3)$, $\nu(\neg\alpha)_{[2]}=\mathsf{b}(\neg\neg\alpha)\leq \mathsf{b}(\alpha)= \nu(\alpha)_{[1]}$, proving that $\nu(\neg\alpha)\in \tilde{\neg}\, \nu(\alpha)$. For $\#\in\{\lor, \land, \to\}$, from condition $(V1)$ one gets that $\nu(\alpha\#\beta)_{[1]}=\mathsf{b}(\alpha\#\beta)=\mathsf{b}(\alpha)\#\mathsf{b}(\beta)= \nu(\alpha)_{[1]}\# \nu(\beta)_{[1]}$. Furthermore, $\nu(\alpha),  \nu(\beta)\in Boo_{n}^{\mathcal{B}}$ if and only if $\nu(\alpha)_{[1]}=\sneg  \nu(\alpha)_{[2]}$ and $\nu(\beta)_{[1]}=\sneg  \nu(\beta)_{[2]}$, or equivalently, $\mathsf{b}(\neg\alpha)=\sneg \mathsf{b}(\alpha)$ and $\mathsf{b}(\neg\beta)=\sneg \mathsf{b}(\beta)$. From Proposition~\ref{B8-bival}, this implies that $\mathsf{b}(\neg(\alpha\#\beta))=\sneg \mathsf{b}(\alpha\#\beta)$, that is, $\nu(\alpha\#\beta)\in Boo_{n}^{\mathcal{B}}$. With all of this, we find that, regardless of the values of $\nu(\alpha)$ and $\nu(\beta)$, $\nu(\alpha\#\beta)\in  \nu(\alpha)\,\tilde{\#}\, v(\beta)$.

Now, we need only to prove that $\nu$ is in $\mathcal{F}_{C_{n}}^{\mathcal{B}}$. From the fact that $\nu$ is a homomorphism, $\nu(\alpha\wedge\neg\alpha)\in  \nu(\alpha)\,\tilde{\land}\, \nu(\neg\alpha)$. Moreover, $\nu(\alpha\wedge\neg\alpha)_{[2]}=\mathsf{b}(\neg(\alpha\wedge\neg\alpha))=\mathsf{b}(\alpha^{1})= \nu(\alpha)_{[3]}$, what proves the first condition for being in $\mathcal{F}_{C_{n}}^{\mathcal{B}}$. From the definition of $\nu$, $\nu(\alpha^{1})_{[1]}=\mathsf{b}(\alpha^{1})= \nu(\alpha)_{[3]}$. From property $(V5)$, $\nu(\alpha^{1})_{[2]}=\mathsf{b}(\neg(\alpha^{1}))=\mathsf{b}(\alpha)\land\mathsf{b}(\neg\alpha)= \nu(\alpha)_{[1]}\land\nu(\alpha)_{[2]}$. For $3\leq k\leq n$, $\nu(\alpha^{1})_{[k]}=\mathsf{b}((\alpha^{1})^{k-2})=\mathsf{b}(\alpha^{k-1})= \nu(\alpha)_{[k+1]}$. Finally, we have from $(V4)_{n}$ that $ \nu(\alpha^{1})_{[n+1]}=\mathsf{b}(\alpha^{n})=\sneg (\mathsf{b}(\alpha^{n-1})\land\mathsf{b}(\neg(\alpha^{n-1})))$. From $(V5)$, $\mathsf{b}(\neg(\alpha^{n-1}))=\mathsf{b}(\alpha^{n-2})\land\mathsf{b}(\neg(\alpha^{n-2}))$, and proceeding recursively, one obtains that $ \nu(\alpha^{1})_{[n+1]}=(\mathsf{b}(\alpha)\land\mathsf{b}(\neg\alpha)\land{\sim}\bigwedge_{i=1}^{n-1}\mathsf{b}(\alpha^{i}))=\sneg \bigwedge_{i=1}^{n+1} \nu(\alpha)_{[i]}$, hence the second condition for $\mathcal{F}_{C_{n}}^{\mathcal{B}}$ is validated. For any $\#\in\{\lor, \land, \to\}$, from $(V6)_{n}$ we find that $\mathsf{b}(\alpha^{(n)})\land\mathsf{b}(\beta^{(n)})\leq\mathsf{b}((\alpha\#\beta)^{(n)})$, that is, $ \nu(\alpha^{(n)}\land\beta^{(n)})_{[1]}= \nu(\alpha^{(n)})_{[1]}\land\nu(\beta^{(n)})_{[1]}\leq  \nu((\alpha\#\beta)^{(n)})_{[1]}$, and therefore $ \nu(\alpha^{(n)}\land\beta^{(n)})_{[1]}\to\nu((\alpha\#\beta)^{(n)})_{[1]}= \nu((\alpha^{(n)}\land\beta^{(n)})\to(\alpha\#\beta)^{(n)})_{[1]}=1$, which is equivalent to $ \nu((\alpha^{(n)}\land\beta^{(n)})\to(\alpha\#\beta)^{(n)})\in D_{n}^{\mathcal{B}}$.

Clearly, $\mathsf{b}(\alpha)=1$ if, and only if, $ \nu(\alpha)_{[1]}=1$, which is in turn equivalent to $ \nu(\alpha)\in D_{n}^{\mathcal{B}}$.
\end{proof}

\

\noindent
From the previous lemmas, completeness of $C_n$ w.r.t. restricted swap structures can be easily proved.

\begin{theorem} [Soundness and Completeness of $C_n$ w.r.t. $\matGS_{C_n}$]  \label{sound-compl-GS-Cn} \ \\
	Let $\Gamma \cup \{\varphi\} \subseteq \bF(\Sigma,\matV)$. Then: $\Gamma \vdash_{C_n} \varphi$ \ iff \ $\Gamma \vDash_{\matGS_{C_n}}^\mathsf{RN} \varphi$. 
\end{theorem}


\noindent Finally, we show that the generalization from {\bf 2} to arbitrary Boolean algebras produces, indeed, new semantical scenarios. 

\begin{prop} \label{val-nontriv}
	Let \B\ be a non-trivial Boolean algebra. Then:\\[1mm]
	(i) There exists a \B-valuation $\mathsf{b}$ for $C_n$ such that $\mathsf{b}(\neg\alpha) \neq \sneg \mathsf{b}(\alpha)$ for some $\alpha$, and such that its image $Im(\mathsf{b})=\{\mathsf{b}(\alpha) \ : \ \alpha \in \bF(\Sigma,\matV)\}$ is not contained in $\{0,1\}$.\\[1mm]
	(ii) There exists a  valuation $\nu \in \matF_{C_n}^\B$ such that its image $Im(\nu)=\{\nu(\alpha) \ : \ \alpha \in \bF(\Sigma,\matV)\}$ is not contained neither in $Boo_n^\B$ nor in $B_n$.
\end{prop}
\begin{proof} $(i)$ Fix $n$ and \B, and suppose that a function $\mathsf{b}:\matV\to |\B|$ was defined. By induction on the complexity of $\alpha \in \bF(\Sigma,\matV)$, this function can be extended to a \B-valuation for $C_n$ with the properties required in~(i). The only clause of \B-valuations whose satisfaction is not so immediate is~$(V6)_n$. Indeed, the other clauses can be easily satisfied by defining $\mathsf{b}$ recursively, and still fulfilling the requirements of~(i). However, clause~$(V2)$ must be additionally restricted when defining the values of $\mathsf{b}(\neg(\alpha\#\beta))$ and $\mathsf{b}((\alpha\#\beta)^i)$ for $1 \leq i \leq n-1$ (if $n \geq 2$) in order to guarantee the satisfaction of~$(V6)_n$, as we shall see.   
	Thus, assume that $\mathsf{b}((\alpha)^{(n)})$ and $\mathsf{b}((\beta)^{(n)})$ were defined, which presupposes that $\mathsf{b}(\alpha^i)$, $\mathsf{b}(\neg(\alpha^i))$, $\mathsf{b}(\beta^i)$ and $\mathsf{b}(\neg(\beta^i))$  are already defined for $0 \leq i \leq n$. Let $a_0:=\mathsf{b}((\alpha)^{(n)}) \land \mathsf{b}((\beta)^{(n)})$ and $a_1:=\mathsf{b}(\alpha \#\beta)=\mathsf{b}(\alpha) \# \mathsf{b}(\beta)$ be defined according to~$(V1)$.\\[1mm]
	$(i.1)$ If $n=1$, let $a_2:=\mathsf{b}(\neg(\alpha \# \beta))$ be such that~(1)~$a_1 \lor a_2=1$, and~(2.1)~$\sneg (a_1 \land a_2) \geq a_0$. Observe that $\mathsf{b}((\alpha \# \beta)^1) = \sneg (a_1 \land a_2) = \mathsf{b}((\alpha \# \beta)^{(1)})$.
	\\[1mm]
	$(i.2)$ If $n=2$, let $a_2:=\mathsf{b}(\neg(\alpha \# \beta))$ satisfying~$(i.1)$(1) and let $a_3:=\mathsf{b}((\alpha\#\beta)^1)$ be such that:~(2)~$(a_1 \land a_2) \lor a_3=1$, and~(3.2)~$\sneg(a_1 \land a_2) \land a_3 \geq a_0$.  Observe that $\mathsf{b}(\neg(\alpha\#\beta)^1)=a_1 \land a_2$, $\mathsf{b}((\alpha \# \beta)^2) = \sneg (a_1 \land a_2 \land a_3)$ and $\mathsf{b}((\alpha \# \beta)^{(2)})=\sneg(a_1 \land a_2) \land a_3$.
	\\[1mm]
	$(i.3)$ If $n=3$, let $a_2:=\mathsf{b}(\neg(\alpha \# \beta))$ satisfying~$(i.1)$(1), $a_3:=\mathsf{b}((\alpha\#\beta)^1)$ satisfying~$(i.2)$(2),  and $a_4=\mathsf{b}((\alpha \# \beta)^2)$ be such that:~(3)~$(a_1 \land a_2 \land a_3) \lor a_4=1$, and~(4.3)~$\sneg(a_1 \land a_2) \land (a_3 \land a_4) \geq a_0$.  Observe that $\mathsf{b}(\neg(\alpha\#\beta)^1)=a_1 \land a_2$, $\mathsf{b}(\neg(\alpha \# \beta)^2) = a_1 \land a_2 \land a_3$, $\mathsf{b}(\neg(\alpha\#\beta)^3)=\sneg(a_1 \land \ldots \land a_4)$ and $\mathsf{b}((\alpha \# \beta)^{(3)})=\sneg(a_1 \land a_2) \land (a_3 \land a_4)$.\\[1mm]
	\vdots\\[1mm]
	$(i.n)$ Define  $a_2:=\mathsf{b}(\neg(\alpha \# \beta))$ satisfying~$a_1 \lor a_2=1$, $a_k:=\mathsf{b}((\alpha\#\beta)^{k-2})$ satisfying $\big(\bigwedge_{j=1}^{k-1} a_j\big) \lor a_k=1$, for $3 \leq k \leq n$, and let $a_{n+1}:= \mathsf{b}((\alpha\#\beta)^{n-1})$ be such that $\big(\bigwedge_{j=1}^n a_j\big) \lor a_{n+1}=1$, and $\sneg(a_1 \land a_2) \land \big(\bigwedge_{j=3}^{n+1} a_j\big) \geq a_0$. 
	By definition, $\mathsf{b}$ satisfies clause~$(V6)_n$. Observe that it is always possible to define $a_2, \ldots, a_n$ satisfying the requirements of~$(i.n)$, by taking $a_k:=\sneg \bigwedge_{j=1}^{k-1} a_j$ for $2 \leq k \leq n+1$. In order to guarantee the requirements of~$(i)$ for $\mathsf{b}$ it is enough to consider $\mathsf{b}(\neg p) \neq \sneg \mathsf{b}(p)$ for at least one propositional variable $p$. In addition, taking $\mathsf{b}(p) \not \in \{0,1\}$ for at least one $p \in \matV$ guarantees that $Im(\mathsf{b}) \not\in\{0,1\}$.\\[1mm]
	$(ii)$ Let $\mathsf{b}$ be a \B-valuation for $C_n$ constructed as in~$(i)$. Then, the function $\nu:\bF(\Sigma,\matV) \to B_n^\B$ obtained from $\mathsf{b}$ as in Lemma~\ref{lem-compl-Cn-B} satisfies the requirements of~$(ii)$.
\end{proof}

\section{Counting snapshots}\label{Counting_snapshots}

As models of $C_n$, at least as long as we take into consideration the restrictions imposed over homomorphisms, the multialgebras $\mathcal{A}_{C_{n}}^{\mathcal{B}}$ have an important role to play, model-theoretically speaking, in the study of da Costa's hierarchy. Although easily defined, the somewhat combinatorial way in which its elements are constructed leads to a complex structure. To show how one could analyze the intricacies of $B_{n}^{\mathcal{B}}$, and start an algebraic study of $\mathcal{A}_{C_{n}}^{\mathcal{B}}$, we prove here that, if $\mathcal{B}$ is a finite Boolean algebra with $2^{m}$ elements, then $B_{n}^{\mathcal{B}}$ has $(n+2)^{m}$ elements. We begin by noticing the following relationship between $B_{n+1}^{\mathcal{B}}$ and $B_{n}^{\mathcal{B}}$, valid for any $n\geq 1$:
\[B_{n+1}^{\mathcal{B}}=\{(a_{[1]}, \ldots  , a_{[n+2]})\in |\mathcal{B}|^{n+2} \ : \  (a_{[1]}, \ldots  , a_{[n+1]})\in B_{n}^{\mathcal{B}}\text{ and }a_{[n+2]}\vee\bigwedge_{i=1}^{n+1}a_{[i]}=1\}.\]
Since all finite, non-trivial Boolean algebras are isomorphic to the powerset of a finite set (their set of atoms), for simplicity we assume here that any finite Boolean algebra with $2^{m}$ elements (for $m\geq 1$) is precisely the powerset $\mathcal{P}(\textbf{m})$ of the prototypical set of $m$ elements, $\textbf{m}=\{0, 1, \ldots  , m-1\}$. Then, an element $a$ of $\mathcal{P}(\textbf{m})$ is said to have order $0\leq k \leq m$ if it is a subset of $\textbf{m}$ with $k$ elements. Recalling that $\binom{m}{k}$ denotes the binomial coefficient {\em $m$ choose $k$}, a simple combinatorial argument shows $\mathcal{P}(\textbf{m})$ has $\binom{m}{0}=1$ elements of order $0$ (namely $\emptyset$, also denoted by $0$), $\binom{m}{1}=m$ elements of order $1$ and, inductively, $\binom{m}{k}$ elements of order $k$.

\begin{lemma}\label{lem-tech-1}
For an element $a$ of $\mathcal{P}(\textbf{m})$ of order $k$, there are $\binom{k}{p}$ elements $b$ such that $a\vee b=1$ and $a\wedge b$ has order $p\leq k$, and $\binom{m-k}{q}$ elements $c$ such that $a\wedge c=0$ and $a\vee c$ has order $q\geq k$.
\end{lemma}

\begin{proof}
We will only prove the first statement, being the second analogous. If $a\vee b=1$, we have $\textbf{m}\setminus a=b\setminus a$; furthermore, if $a\wedge b$ has order $p$, this means $a\cap b$ has $p$ elements and therefore may equal any one of $\binom{k}{p}$ possible sets. Since there is one possibility for $b\setminus a$, and $\binom{k}{p}$ possibilities for $a\cap b$, this gives us a total of $\binom{k}{p}$.
\end{proof}

\begin{lemma}\label{lem-tech-2}
For $m\in\mathbb{N}$, $p\leq m$ and $x\in\mathbb{R}$, $\displaystyle \sum_{j=p}^{m}\binom{j}{p}\binom{m}{j}x^{m-j}=\binom{m}{p}(x+1)^{m-p}$.
\end{lemma}

\begin{proof}
From the binomial theorem,
\[\sum_{j=p}^{m}\binom{j}{p}\binom{m}{j}x^{m-j}=\sum_{j=p}^{m}\frac{j!}{p!(j-p)!}\frac{m!}{j!(m-j)!}x^{m-j}=\sum_{j=p}^{m}\frac{1}{p!}\frac{m!}{(j-p)!(m-j)!}x^{m-j}=\]
\[\sum_{j=p}^{m}\frac{m!}{p!(m-p)!}\frac{(m-p)!}{(j-p)!(m-j)!}x^{m-j}=\binom{m}{p}\sum_{i=0}^{m-p}\frac{(m-p)!}{i!((m-p)-i)!}x^{(m-p)-i}=\]
\[\binom{m}{p}\sum_{i=0}^{m-p}\binom{m-p}{i}x^{(m-p)-i}1^{i}=\binom{m}{p}(x+1)^{m-p}.\]
\end{proof}

\begin{lemma}\label{lem-tech-3}
If $\mathcal{B}$ is the Boolean algebra with $2^{m}$ elements, for $n\geq 1$, $B_{n}^{\mathcal{B}}$ has exactly\\ $\binom{m}{p}(n+1)^{m-p}$ elements $(a_{[1]}, \ldots  , a_{[n+1]})$ with $\bigwedge_{i=1}^{n+1}a_{[i]}$ of order $p\leq m$.
\end{lemma}

\begin{proof}
We proceed by induction on $n$, starting with $n=1$. For an element $a$ of $\mathcal{B}$ with order $k\geq p$, there are $\binom{k}{p}$ possible $b$ such that $(a,b)$ is in $B_{1}^{\mathcal{B}}$ (i.e. $a\vee b=1$) and $a\wedge b$ has order $p$ from Lemma \ref{lem-tech-1}; of course, if $k\leq p$ there are none. Given that $\mathcal{B}$ has $\binom{m}{k}$ elements $a$ of order $k$, the total number of pairs $(a,b)$ in $B_{1}^{\mathcal{B}}$ with $a\wedge b$ of order $p$ becomes $\sum_{k=p}^{m}\binom{k}{p}\binom{m}{k}$, equal to $\binom{m}{p}2^{m-p}$ by Lemma \ref{lem-tech-2} with $x=1$.

For induction hypothesis, suppose the lemma holds for $B_{n}^{\mathcal{B}}$. There are, then, $\binom{m}{k}(n+1)^{m-k}$ elements $(a_{[1]}, \ldots  , a_{[n+1]})$ of $B_{n}^{\mathcal{B}}$ with $\bigwedge_{i=1}^{n+1}a_{[i]}$ of order $k$. From Lemma \ref{lem-tech-1}, we find there are $\binom{k}{p}$ values for $a_{[n+2]}$ satisfying, first of all, that $(a_{[1]}, \ldots  , a_{[n+2]})\in B_{n+1}^{\mathcal{B}}$ (what amounts to $a_{[n+2]}\vee\bigwedge_{i=1}^{n+1}a_{[i]}=1$); and that $\bigwedge_{i=1}^{n+2}a_{[i]}$ is an element of $\mathcal{B}$ of order $p$, adding up to a total of $\sum_{k=p}^{m}\binom{k}{p}\binom{m}{k}(n+1)^{m-k}=\binom{m}{p}(n+2)^{m-p}$, according to Lemma \ref{lem-tech-2} once one sets $x=n+1$. This finishes our proof.
\end{proof}

\begin{theorem}
If $\mathcal{B}$ is a Boolean algebra with $2^{m}$ elements, there are $(n+2)^{m}$ snapshots in $B_{n}^{\mathcal{B}}$.
\end{theorem}

\begin{proof}
From Lemma \ref{lem-tech-3}, $B_{n}^{\mathcal{B}}$ has $\binom{m}{0}(n+1)^{m-0}$ snapshots $(a_{[1]}, \ldots  , a_{[n+1]})$ with $\bigwedge_{i=1}^{n+1}a_{[i]}$ of order $0$, $\binom{m}{1}(n+1)^{m-1}$ snapshots with $\bigwedge_{i=1}^{n+1}a_{[i]}$ of order $1$ and so on. From Lemma \ref{lem-tech-2}, this adds up to $\sum_{p=0}^{m}\binom{m}{p}(n+1)^{m-p}=(n+2)^{m}$.
\end{proof}

If $(a, b)$ is a pair on $|\mathcal{B}|^{2}$, it lies in $D_{1}^{\mathcal{B}}$ iff $a=1$ and $a\vee b=1$, meaning $b$ may assume any value in $\mathcal{B}$; from this, we deduce $D_{1}^{\mathcal{B}}$ has as many elements as $\mathcal{B}$ itself. Inductively, by using that
\[D_{n+1}^{\mathcal{B}}=\{(1, a_{[1]}, \ldots  , a_{[n+1]})\in B_{n+1}^{\mathcal{B}}: (a_{[1]}, \ldots  , a_{[n+1]})\in B_{n}^{\mathcal{B}}\},\]
for $n\geq 1$, $D_{n+1}^{\mathcal{B}}$ has as many snapshots as $B_{n}^{\mathcal{B}}$.

\begin{theorem}
If $\mathcal{B}$ has $2^{m}$ elements, $D_{n}^{\mathcal{B}}$ and $Boo_{n}^{\mathcal{B}}$ have, respectively $(n+1)^{m}$ and $2^{m}$ elements.
\end{theorem}

For completeness sake, we may mention the case in which $\mathcal{B}$ is infinite, e.g. of cardinality $\kappa$. First of all, $Boo_{n}^{\mathcal{B}}\subseteq B_{n}^{\mathcal{B}}$, and given the former is isomorphic to $\mathcal{B}$ we obtain $B_{n}^{\mathcal{B}}$ has cardinality at least $\kappa$. At the same time, $B_{n}^{\mathcal{B}}\subseteq|\mathcal{B}|^{n+1}$, the latter being too of cardinality $\kappa$ from the fact this is an infinite cardinal. Continuing this line of thought, we obtain $B_{n}^{\mathcal{B}}$, $Boo_{n}^{\mathcal{B}}$ and $D_{n}^{\mathcal{B}}$ are all of cardinality $\kappa$.

We therefore reach the conclusion that the number of snapshots increases, and increases rather quickly, with both the $n$ of $C_{n}$ and the cardinality of $\mathcal{B}$. To give an example of the complexity of $B_{n}^{\mathcal{B}}$, take the four-valued Boolean algebra $\mathcal{B}_{4}$ as the power-set of $\{a,b\}$, for simplicity of notation. We will also denote $\emptyset$ by $0$, and $\{a,b\}$ by $1$. Then, $B_{1}^{\mathcal{B}_{4}}$ has $9$ snapshots: 
\begin{enumerate}
\item designated and Boolean ones, $(1, 0)$;
\item designated, but not Boolean, ones, $(1, \{a\})$, $(1, \{b\})$ and $(1, 1)$;
\item Boolean, but undesignated, ones $(0,1)$, $(\{a\}, \{b\})$ and $(\{b\}, \{a\})$;
\item not Boolean and undesignated ones, $(\{a\}, 1)$ and $(\{b\}, 1)$.
\end{enumerate}
Meanwhile, $B_{2}^{\mathcal{B}_{4}}$ has $16$ snapshots:
\begin{enumerate}
\item designated and Boolean ones, $(1, 0, 1)$;
\item designated, but not Boolean, ones, $(1, 1, 0)$, $(1, 1, \{a\})$, $(1, 1, \{b\})$, $(1, 1, 1)$, $(1, \{a\}, \{b\})$, $(1, \{b\}, \{a\})$, $(1, \{a\}, 1)$ and $(1, \{b\}, 1)$;
\item Boolean, but undesignated ones, $(0, 1, 1)$, $(\{a\}, \{b\}, 1)$ and $(\{b\}, \{a\}, 1)$;
\item not Boolean and undesignated ones, $(\{a\}, 1, 1)$, $(\{a\}, 1, \{b\})$, $(\{b\}, 1, 1)$ and $(\{b\}, 1, \{a\})$.

\end{enumerate}

\section{Category of swap structures}\label{Category}

Let $C$ be a class of RNmatrices over some signature $\Theta$. How to endow it with morphisms so that the resulting object is a category? Of course, this depends on what one wishes to achieve, but a general method seems reasonably within reach: after all, an RNmatrix has three components, a $\Theta$-multialgebra $\mathcal{A}$, a subset $D$ of its universe, and a set $\mathcal{F}$ of homomorphisms $\nu:\textbf{F}(\Theta, \mathcal{V})\rightarrow\mathcal{A}$; it stands to scrutiny that an ideal morphism on the category with $C$ as objects should respect all three of these elements. That is, a morphism on $C$, between $\mathcal{M}=(\mathcal{A}, D, \mathcal{F})$ and $\mathcal{M}^{*}=(\mathcal{A}^{*}, D^{*}, \mathcal{F}^{*})$ should be: ~(1)~ a $\Theta$-homomorphism $h:\mathcal{A}\rightarrow\mathcal{A}^{*}$; ~(2)~ which maps designated elements unto designated elements, i.e. $h[D]\subseteq D^{*}$; ~(3)~ which is absorbed by restricted valuations, meaning that for any $\nu\in\mathcal{F}$, $h\circ\nu\in\mathcal{F}^{*}$. 

\begin{theorem}\label{category_of_RNmatrices}
A class $C$ of RNmatrices, equipped with the morphisms defined above, becomes a category $\mathcal{C}$.
\end{theorem}

\begin{proof}
Suppose $h:\mathcal{M}_{1}\rightarrow\mathcal{M}_{2}$ and $g:\mathcal{M}_{2}\rightarrow\mathcal{M}_{3}$ are morphisms as previously defined. Then $g\circ h$ is a $\Theta$-homomorphism since the composition of $\Theta$-homomorphisms returns $\Theta$-homomorphisms. Given $h[D_{1}]\subseteq D_{2}$ and $g[D_{2}]\subseteq D_{3}$, $g\circ h[D_{1}]=g[h[D_{1}]]\subseteq g[D_{2}]\subseteq D_{3}$; and, for $\nu\in\mathcal{F}_{1}$, $h\circ\nu\in\mathcal{F}_{2}$ given that $h$ is a morphism, and therefore $(g\circ h)\circ\nu=g\circ (h\circ\nu)\in\mathcal{F}_{3}$ given that $g$ is also a morphism. All of this of course implies that $g\circ h$ remains a morphism, and therefore the composition of morphisms returns morphisms.

Associativity of the composition of morphisms comes from the fact these are functions, and the identity morphisms are precisely the identity functions, which are trivially seem to satisfy all necessary requirements.
\end{proof}

So we define now a category of restricted swap structures for each $C_{n}$, in order to display some of the nice model-theoretical and categorical properties of the RNmatrices $\mathcal{RM}^{\mathcal{B}}_{C_{n}}$. Let $\textbf{RSwap}_{C_{n}}$ be the category constructed from the class of RNmatrices $\mathcal{RM}_{C_{n}}^{\mathcal{B}}$, for $\mathcal{B}$ a non-trivial Boolean algebra. More explicitly, $\textbf{RSwap}_{C_{n}}$ is the category with: as objects, the (proper) class of (full) restricted swap structures $\mathcal{A}_{C_{n}}^{\mathcal{B}}$, for every Boolean algebra $\mathcal{B}$. As morphisms from $\mathcal{A}_{C_{n}}^{\mathcal{B}_{1}}$ to $\mathcal{A}_{C_{n}}^{\mathcal{B}_{2}}$ (for Boolean algebras $\mathcal{B}_{1}$ and $\mathcal{B}_{2}$), all homomorphisms $h:\mathcal{A}_{C_{n}}^{\mathcal{B}_{1}}\rightarrow\mathcal{A}_{C_{n}}^{\mathcal{B}_{2}}$ of $\Sigma$-multialgebras such that:\footnote{One can actually prove that, in the case of $\textbf{RSwap}_{C_{n}}$, is not necessary to assume $h[D_{n}^{\mathcal{B}_{1}}]\subseteq D_{n}^{\mathcal{B}_{2}}$, given that the nature of the homomorphisms in $\mathcal{F}_{C_{n}}^{\mathcal{B}_{1}}$ and $\mathcal{F}_{C_{n}}^{\mathcal{B}_{2}}$ already implies this property. We still maintain the first condition for both homogeneity and simplicity.} ~(1)~ $h[D_{n}^{\mathcal{B}_{1}}]\subseteq D_{n}^{\mathcal{B}_{2}}$; and~(2)~for any $\nu$ in $\mathcal{F}_{C_{n}}^{\mathcal{B}_{1}}$, $h\circ\nu$ is in $\mathcal{F}_{C_{n}}^{\mathcal{B}_{2}}$.

\begin{figure}[H]
\centering
\begin{tikzcd}
    \mathcal{A}_{C_{n}}^{\mathcal{B}_{1}} \arrow{rr}{h}  &  & \mathcal{A}_{C_{n}}^{\mathcal{B}_{2}}\\
    & \textbf{F}(\Sigma, \mathcal{V}) \arrow{ul}{\nu} \arrow{ur}{h\circ\nu} & 
  \end{tikzcd}
  \caption*{If $h$ is a morphism from $\mathcal{A}^{\mathcal{B}_{1}}_{C_{n}}$ to $\mathcal{A}^{\mathcal{B}_{2}}_{C_{n}}$, and $\nu\in\mathcal{F}_{C_{n}}^{\mathcal{B}_{1}}$, $h\circ\nu$ is in $\mathcal{F}_{C_{n}}^{\mathcal{B}_{2}}$}
\end{figure}

\begin{prop}
$\textbf{RSwap}_{C_{n}}$ is a category.
\end{prop}

\begin{proof}
Follows from Theorem \ref{category_of_RNmatrices}.
\end{proof}

\begin{prop}\label{hom._of_Bool_are_morp.}
Given Boolean algebras $\mathcal{B}_{1}$ and $\mathcal{B}_{2}$ and a homomorphism $f:\mathcal{B}_{1}\rightarrow\mathcal{B}_{2}$ of Boolean algebras, $h:\mathcal{A}^{\mathcal{B}_{1}}_{C_{n}}\rightarrow\mathcal{A}^{\mathcal{B}_{2}}_{C_{n}}$ defined by $h(z)=(f(z_{[1]}), \ldots  , f(z_{[n+1]}))$, for every $z=(z_{[1]}, \ldots  , z_{[n+1]})\in B_{n}^{\mathcal{B}_{1}}$, is a morphism of $\textbf{RSwap}_{C_{n}}$.
\end{prop}

\begin{proof}
Remember $z$ is a snapshot on $Boo_{n}^{\mathcal{B}_{1}}$ iff it has the form $(a, {\sim}a, 1, \ldots  , 1)$ for some $a\in\mathcal{B}_{1}$. With this, $h(z)=(f(a), {\sim}f(a), 1, \ldots  , 1)$ is an element of $Boo_{n}^{\mathcal{B}_{2}}$, therefore implying $h$ preserves Boolean elements. Now take snapshots $w, z\in B_{n}^{\mathcal{B}_{1}}$. If either $w$ or $z$ is not Boolean and $u\in w\tilde{\#}z$ (\textit{id est} $u_{[1]}=w_{[1]}\# z_{[1]}$), $f$ being a homomorphism of Boolean algebras implies that $h(u)=(f(u_{[1]}), f(u_{[2]}), \ldots  , f(u_{[n+1]}))$ equals $(f(w_{[1]})\# f(z_{[1]}), f(u_{[2]}), \ldots  , f(u_{[n+1]}))$. We therefore deduce that $h(u)\in h(w)\tilde{\#}h(z)$ since $h(w)_{[1]}=f(w_{[1]})$ and $h(z)_{[1]}=f(z_{[1]})$. The remaining case, on which $w, z\in Boo_{n}^{\mathcal{B}_{1}}$, $u$ being in $w\tilde{\#}z$ implies, first of all, that $u$ is also Boolean (from the definition of $\tilde{\#}$), and so is $h(u)$ given that $h$ preserves Boolean snapshots. Second, $u_{[1]}=w_{[1]}\# z_{[1]}$ and thus $h(u)_{[1]}=h(w)_{[1]}\# h(z)_{[1]}$, leading one to $h(u)\in h(w)\tilde{\#} h(z)$. Finally, if $z$ is in $B_{n}^{\mathcal{B}_{1}}$ and $w$ is in $\tilde{\neg}\,z$ (equivalent to $w_{[1]}=z_{[2]}$ and $w_{[2]}\leq z_{[1]}$), $h(w)$ equals $(f(w_{[1]}), \ldots  , f(w_{[n+1]}))$ and analogously for $h(z)$, leading to $h(w)_{[1]}=h(z)_{[2]}$ and $h(w)_{[2]}\leq h(z)_{[1]}$, that is $h(w)\in\tilde{\neg}\, h(z)$, what finishes proving that $h$ is a homomorphism.

If $z=(1, z_{[2]}, \ldots  , z_{[n+1]})$ is a designated element of $\mathcal{A}_{C_{n}}^{\mathcal{B}_{1}}$, $h(z)=(f(1), f(z_{[2]}), \ldots  , f(z_{[n+1]}))$, which equals $(1, f(z_{[2]}), \ldots  , f(z_{[n+1]}))$ since $f$ is a homomorphism of Boolean algebras. Of course $h(z)$ is then also a designated element, and so $h[D_{n}^{\mathcal{B}_{1}}]\subseteq D_{n}^{\mathcal{B}_{2}}$.

Now, $\nu\in\mathcal{F}_{C_{n}}^{\mathcal{B}_{1}}$ whenever, for any formulas $\alpha$ and $\beta$: $\nu(\alpha\wedge\neg\alpha)_{[2]}=\nu(\alpha)_{[3]}$,
\[\nu(\alpha^{1})=(\nu(\alpha)_{[3]}, \nu(\alpha)_{[1]}\wedge\nu(\alpha)_{[2]}, \nu(\alpha)_{[4]}, \ldots  , \nu(\alpha)_{[n+1]}, {\sim}(\bigwedge_{i=1}^{n+1}\nu(\alpha)_{[i]}))\]
and $\nu((\alpha^{(n)}\wedge\beta^{(n)})\rightarrow(\alpha\#\beta)^{(n)})\in D_{n}^{\mathcal{B}_{1}}$, for any $\#\in\{\vee, \wedge, \rightarrow\}$. Quite clearly $h\circ\nu$ remains a homomorphism, so to prove $h$ is in $\textbf{RSwap}_{C_{n}}$ we have yet to prove that this homomorphism lies in $\mathcal{F}_{C_{n}}^{\mathcal{B}_{2}}$. By definition of $h$, $h(\nu(\alpha\wedge\neg\alpha))_{[2]}=f(\nu(\alpha\wedge\neg\alpha)_{[2]})$, and from the fact that $\nu$ lies in $\mathcal{F}_{C_{n}}^{\mathcal{B}_{1}}$ one obtains $f(\nu(\alpha)_{[3]})=h(\nu(\alpha))_{[3]}$. Since $h(\nu(\alpha^{1}))_{[1]}=f(\nu(\alpha)_{[3]})=h(\nu(\alpha))_{[3]}$, $h(\nu(\alpha^{1}))_{[2]}=f(\nu(\alpha)_{[1]}\wedge\nu(\alpha)_{[2]})=h(\nu(\alpha))_{[1]}\wedge h(\nu(\alpha))_{[2]}$, $h(\nu(\alpha^{1}))_{[i]}=f(\nu(\alpha)_{[i+1]})=h(\nu(\alpha))_{[i+1]}$ (for $3\leq i\leq n$) and $h(\nu(\alpha^{1}))_{[n+1]}=f({\sim}\bigwedge_{i=1}^{n+1}\nu(\alpha)_{[i]})={\sim}\bigwedge_{i=1}^{n+1}h(\nu(\alpha))_{[i]}$, $h$ satisfies the second condition for being in $\mathcal{F}_{C_{n}}^{\mathcal{B}_{2}}$. Finally, $z$ being designated (i.e. $z_{[1]}=1$) implies, by using that $h(z)$ equals $(f(1), f(z_{[2]}), \ldots  , f(z_{[n+1]}))$, that $h(z)$ is also designated, so $h(\nu((\alpha^{(n)}\wedge\beta^{(n)})\rightarrow(\alpha\#\beta)^{(n)}))$ is always on $D_{n}^{\mathcal{B}_{2}}$.
\end{proof}

Motivated by Proposition~\ref{hom._of_Bool_are_morp.}, the identity morphism of $\mathcal{A}_{C_{n}}^{\mathcal{B}}$ on $\textbf{RSwap}_{C_{n}}$ may be written, on an arbitrary $z\in B_{n}^{\mathcal{B}}$, as $(Id_{\mathcal{B}}(z_{[1]}), \ldots  , Id_{\mathcal{B}}(z_{[n+1]}))$, for $Id_{\mathcal{B}}$ the identity homomorphism on $\mathcal{B}$, being therefore a particular case of the construction shown above; we set now to show that one actually has that all morphisms of the aforementioned category are of the described form.

\subsection{Morphisms of $\textbf{RSwap}_{C_{n}}$}

For a function $h:B^{\mathcal{B}_{1}}_{n}\rightarrow B^{\mathcal{B}_{2}}_{n}$ we may write, for an arbitrary snapshot $z\in B_{n}^{\mathcal{B}_{1}}$, $h(z)=(h_{1}(z), \ldots  , h_{n+1}(z))$ where, for $1\leq i\leq n+1$, $h_{i}$ is a function from $B_{n}^{\mathcal{B}_{1}}$ to $\mathcal{B}_{2}$.\footnote{Technically, $h_i=\pi_i \circ h$ where $\pi_i$ is the $ith$ projection from $B_{n}^{\mathcal{B}_{2}}$ to $|\mathcal{B}_{2}|$ for $1 \leq i \leq n+1$. Note that $h_i(z)=h(z)_{[i]}$, according to the notation previously adopted for snapshots.} Then, $h$ is a $\Sigma$-homomorphism iff, for $w, z\in B_{n}^{\mathcal{B}_{1}}$, $u\in w\tilde{\#}z$ and $v\in\tilde{\neg}\,z$, $h(u)\in h(w)\tilde{\#}h(z)$ and $h(v)\in \tilde{\neg}\,h(z)$, itself equivalent to $h_{1}(u)=h_{1}(w)\# h_{1}(z)$, and $h_{1}(v)=h_{2}(z)$ and $h_{2}(v)\leq h_{1}(z)$. 

Assuming now that $h$ is indeed a homomorphism, we prove that the function $g:|\mathcal{B}_{1}|\rightarrow |\mathcal{B}_{2}|$ defined by $g(a)=h_{1}((a, {\sim}a, 1, \ldots  , 1))$, for any $a\in\mathcal{B}_{1}$, satisfies $h_{1}(z)=g(z_{[1]})$, for any $z\in B_{n}^{\mathcal{B}_{1}}$. Indeed, take an arbitrary snapshot $z=(z_{[1]}, \ldots  , z_{[n+1]})$ in $B_{n}^{\mathcal{B}_{1}}$ and $z'=(z_{[1]}, {\sim}z_{[1]}, 1, \ldots  , 1)$. By definition of $g$ we have $h_{1}(z')=g(z_{[1]})$, so it remains to show that $h_{1}(z)=h_{1}(z')$. If one recalls that $t_{0}^{n}=(1, 1, 0, \ldots , 1)$, $z\,\tilde{\wedge}\, t_{0}^{n}$ and $z'\,\tilde{\wedge}\, t_{0}^{n}$ both coincide with $\{w\in B_{n}^{\mathcal{B}_{1}} \ : \ w_{[1]}=z_{[1]}\}$ (since $t^{n}_{0}\notin Boo_{n}^{\mathcal{B}_{1}}$), and therefore $z, z'\in z\,\tilde{\wedge}\, t_{0}^{n}$. Since $h$ is a homomorphism, $z, z'\in z\,\tilde{\wedge}\, t^{n}_{0}$ implies that $h_{1}(z)=h_{1}(z)\wedge h_{1}(t_{0}^{n})=h_{1}(z')$, what proves that $h_{1}(z)=g(z_{[1]})$. Additionally, since for all $a, b\in |\mathcal{B}_{1}|$, $(a, {\sim}a, 1, \ldots  , 1)\,\tilde{\#}\,(b, {\sim} b, 1, \ldots  , 1)=\{(a\# b, {\sim}(a\# b), 1, \ldots  , 1)\}$, for any $\#\in\{\vee, \wedge, \rightarrow\}$, we may also derive that $g(a\# b)=g(a)\# g(b)$. We henceforth write $h(z)=(g(z_{[1]}), h_{2}(z), \ldots  , h_{n+1}(z))$, for any snapshot $z$.

Going even further, we may also define the function $\theta:|\mathcal{B}_{1}|\rightarrow|\mathcal{B}_{2}|$ by $\theta(a)=h_{2}(({\sim}a, a, 1, \ldots  , 1))$, for any $a\in |\mathcal{B}_{1}|$. What we proceed to show is, first of all, that for any snapshot $z$, $h_{2}(z)=\theta(z_{[2]})$. Again, take an arbitrary element $z=(z_{[1]}, z_{[2]}, z_{[3]}, \ldots  , z_{[n+1]})\in B_{n}^{\mathcal{B}_{1}}$ and make $z'=({\sim}z_{[2]}, z_{[2]}, 1, \ldots  , 1)$. We have that $h_{2}(z')=\theta(z_{[2]})$, so it must be shown that $h_{2}(z)=h_{2}(z')$. Since $z$ is a snapshot, $z_{[1]}\vee z_{[2]}=1$ and so ${\sim}z_{[2]}\leq z_{[1]}$, leading us to define  $z^{*}=(z_{[2]}, {\sim}z_{[2]}, 1, \ldots  , 1)$, which satisfies $z^{*}\in \tilde{\neg}\,z\cap \tilde{\neg}\,z'$. From this, $h(z^{*})\in \tilde{\neg}\, h(z)$ and $h(z^{*})\in\tilde{\neg}\, h(z')$, meaning that $h_{1}(z^{*})=h_{2}(z)$ and $h_{1}(z^{*})=h_{2}(z')$ or, in other worlds, $h_{2}(z)=h_{2}(z')$.  From now on, we write $h(z)=(g(z_{[1]}), \theta(z_{[2]}), h_{3}(z), \ldots  , h_{n+1}(z))$.

But we are able to prove $g=\theta$ as well: for $a\in |\mathcal{B}_{1}|$, we define the snapshots $z=({\sim}a, a, 1, \ldots  , 1)$ and $z'=(a, {\sim}a, 1, \ldots  , 1)\in B_{n}^{\mathcal{B}_{1}}$. We have that $z'\in\tilde{\neg}\,z$ (actually $\tilde{\neg}\,z=\{z'\}$ and vice-versa) and therefore $h(z')\in\tilde{\neg}\, h(z)$, implying that $g(a)=\theta(a)$ and $\theta({\sim}a)\leq g({\sim}a)$, the first equation being the one we wanted to prove. We shall now write, given $g=\theta$, $h(z)=(g(z_{[1]}), g(z_{[2]}), h_{3}(z), \ldots  , h_{n+1}(z))$, for an arbitrary snapshot $z$.

We therefore have the following theorem, which summarizes our developments so far.

\begin{theorem}
If $h:\mathcal{A}_{C_{n}}^{\mathcal{B}_{1}}\rightarrow\mathcal{A}_{C_{n}}^{\mathcal{B}_{2}}$ is a $\Sigma$-homomorphism, there exists a function $g:|\mathcal{B}_{1}|\rightarrow|\mathcal{B}_{2}|$ such that $h_{1}(z)=g(z_{[1]})$ and $h_{2}(z)= g(z_{[2]})$, for any $z\in B_{n}^{\mathcal{B}_{1}}$.
\end{theorem}

Now we move to the second and third conditions for being a morphism of $\textbf{RSwap}_{C_{n}}$. From here on out, we assume that $h$ is absorbed by valuations of our RNmatrices, meaning that, for any restricted valuation $\nu:\textbf{F}(\Sigma, \mathcal{V})\rightarrow\mathcal{A}_{C_{n}}^{\mathcal{B}_{1}}$ in $\mathcal{F}_{C_{n}}^{\mathcal{B}_{1}}$, $h\circ\nu$ lies in $\mathcal{F}_{C_{n}}^{\mathcal{B}_{2}}$, and $h[D_{n}^{\mathcal{B}_{1}}]\subseteq D_{n}^{\mathcal{B}_{2}}$.

For any element $z=(1, z_{[2]}, \ldots  , z_{[n+1]})$ of $D_{n}^{\mathcal{B}_{1}}$ we then have that $h(z)=$\\$(g(1), g(z_{[2]}), h_{3}(z_{[3]}), \ldots  , h_{n+1}(z_{[n+1]}))$ is in $D_{n}^{\mathcal{B}_{2}}$, and therefore $g(1)=1$. Even more: for any formula $\alpha$ of $C_{n}$, given a $\nu\in\mathcal{F}_{C_{n}}^{\mathcal{B}_{1}}$ we have that $\nu(\alpha\wedge\neg\alpha\wedge\alpha^{(n)})=F_{n}$ (from Proposition~\ref{strong_negation}). Since $h\circ\nu$ must be in $\mathcal{F}_{C_{n}}^{\mathcal{B}_{2}}$, it follows that $h\circ\nu(\alpha\wedge\neg\alpha\wedge\alpha^{(n)})=F_{n}$, which implies that $h(F_{n})=F_{n}$ and therefore $g(0)=0$. Finally, we can then prove that $g$ is a homomorphism of Boolean algebras: we already know it satisfies $g(a\#b)=g(a)\# g(b)$, for every $\#\in\{\vee, \wedge, \rightarrow\}$, and $g(0)=0$ and $g(1)=1$. From this, for any $a\in |\mathcal{B}_{1}|$, $g(a)\vee g({\sim}a)=g(a\vee{\sim}a)=g(1)=1$ and $g(a)\wedge g({\sim}a)=g(a\wedge{\sim}a)=g(0)=0$, implying that $g({\sim}a)={\sim}g(a)$.

Finally, again for an arbitrary formula $\alpha$ and a restricted valuation $\nu\in\mathcal{F}_{C_{n}}^{\mathcal{B}_{1}}$, the relevant property here will be that $\nu(\alpha^{k})_{[1]}=\nu(\alpha)_{[k+2]}$, for $1\leq k\leq n-1$. Given a snapshot $z\in B_{n}^{\mathcal{B}_{1}}$, take a propositional variable $p$ and a restricted valuation on $\mathcal{A}_{C_{n}}^{\mathcal{B}_{1}}$ such that $\nu(p)=z$, and so $\nu(p^{k})_{[1]}=z_{[k+2]}$. Since $h\circ \nu$ must be a restricted valuation of $B_{n}^{\mathcal{B}_{2}}$, $h(\nu(p^{k}))_{[1]}=h(\nu(p))_{[k+2]}=h(z)_{[k+2]}=h_{k+2}(z)$. On another direction, $h(\nu(p^{k}))_{[1]}=g(\nu(p^{k})_{[1]})=g(z_{[k+2]})$, thus $h_{k+2}(z)=g(z_{[k+2]})$, for any $k\in\{1, \ldots  , n-1\}$. We may summarize what we obtained in the following theorem.

\begin{theorem}\label{morp._are_n-tuples_hom.}
If $h:\mathcal{A}_{C_{n}}^{\mathcal{B}_{1}}\rightarrow\mathcal{A}_{C_{n}}^{\mathcal{B}_{2}}$ is a morphism of $\textbf{RSwap}_{C_{n}}$, there exists a homomorphism $g:\mathcal{B}_{1}\rightarrow\mathcal{B}_{2}$ of Boolean algebras such that $h_{i}(z)=g(z_{[i]})$, for any $z\in B_{n}^{\mathcal{B}_{1}}$ and $i\in\{1, \ldots  , n+1\}$.
\end{theorem}

\subsection{$\textbf{BA}$ and $\textbf{RSwap}_{C_{n}}$ are isomorphic}

In this subsection it will be proven that the category $\textbf{RSwap}_{C_{n}}$ is isomorphic to the category $\textbf{BA}$ of (non-degenerate) Boolean algebras.  

\begin{prop}
Consider, for any Boolean snapshots $(a, {\sim}a, 1, \ldots  , 1), (b, {\sim}b, 1, \ldots  , 1)$ in $Boo_{n}^{\mathcal{B}}$, the following operations: 
$(a, {\sim}a, 1 \ldots  , 1)\#(b, {\sim}b, 1, \ldots  , 1)=(a\#b, {\sim}(a\#b), 1, \ldots  , 1)$,  for $\#\in\{\vee, \wedge, \rightarrow\}$;
${\sim}(a, {\sim}a, 1, \ldots  , 1)=({\sim}a, a, 1, \ldots  , 1)$; $\top=(1, 0, 1, \ldots  , 1)$; and $\bot=(0, 1, 1, \ldots  , 1)$. Then, $Boo_{n}^{\mathcal{B}}$ becomes a Boolean algebra with this structure. Furthermore, the map $\rho:|\mathcal{B}|\rightarrow Boo_{n}^{\mathcal{B}}$, defined by $\rho(a)=(a, {\sim}a, 1, \ldots  , 1)$, is an isomorphism of Boolean algebras.
\end{prop}

We omit the proof of the previous proposition given that it is self-evident. More importantly, notice the operations we have defined and that make $Boo_{n}^{\mathcal{B}}$ into a Boolean algebra are the only ones that make it into a submultialgebra of $\mathcal{A}_{C_{n}}^{\mathcal{B}}$, \textit{id est}, they satisfy $\rho(a)\tilde{\#}\rho(b)=\{\rho(a)\#\rho(b)\}$ and $\tilde{\neg}\,\rho(a)=\{{\sim}\rho(a)\}$, for any $\#\in\{\vee, \wedge, \rightarrow\}$ and $a, b\in\mathcal{B}$ (of course, one has $\sim$ as negation, while the other has $\neg$, but this is mostly notational).

We will now make use of the category $\textbf{BA}$ of non-degenerate Boolean algebras (that is, Boolean algebras with $0\neq 1$), equipped with homomorphisms of Boolean algebras as morphisms. We then define the functors: ~(1)~ $\mathcal{A}_{n}:\textbf{BA}\rightarrow\textbf{RSwap}_{C_{n}}$\label{An}, taking a Boolean algebra $\mathcal{B}$ to $\mathcal{A}_{n}\mathcal{B}=\mathcal{A}_{C_{n}}^{\mathcal{B}}$, and a homomorphism $g:\mathcal{B}_{1}\rightarrow\mathcal{B}_{2}$ to the morphism $\mathcal{A}_{n}g:\mathcal{A}_{C_{n}}^{\mathcal{B}_{1}}\rightarrow\mathcal{A}_{C_{n}}^{\mathcal{B}_{2}}$ such that, for any snapshot $z\in B_{n}^{\mathcal{B}_{1}}$, $\mathcal{A}_{n}g(z)_{[i]}=g(z_{[i]})$, for every $i\in\{1, \ldots  , n+1\}$; (2)~$\textbf{Boo}_{n}$\label{tBoon}, taking $\mathcal{A}_{C_{n}}^{\mathcal{B}}$ to $\mathcal{B}$,\footnote{Equivalently, one could take, through the functor $\textbf{Boo}_{n}$, $\mathcal{A}_{C_{n}}^{\mathcal{B}}$ to the Boolean algebra $Boo_{n}^{\mathcal{B}}$, which is isomorphic to $\mathcal{B}$. Of course, in that case, $\mathcal{A}_{n}$ and $\textbf{Boo}_{n}$ would no longer be an isomorphism of categories, but rather an equivalence.} and a morphism $h:\mathcal{A}_{C_{n}}^{\mathcal{B}_{1}}\rightarrow\mathcal{A}_{C_{n}}^{\mathcal{B}_{2}}$ to the homomorphism of Boolean algebras $\textbf{Boo}_{n}h:\mathcal{B}_{1}\rightarrow\mathcal{B}_{2}$ defined by $\textbf{Boo}_{n}h(a)=h((a, {\sim}a, 1, \ldots  , 1))_{1}$, for any $a\in|\mathcal{B}|$.

\begin{prop}
As defined, $\mathcal{A}_{n}$ and $\textbf{Boo}_{n}$ are, indeed, functors.
\end{prop}

\begin{proof}
As proved in Proposition~\ref{hom._of_Bool_are_morp.}, for Boolean algebras $\mathcal{B}_{1}$ and $\mathcal{B}_{2}$, and a homomorphism $g:\mathcal{B}_{1}\rightarrow\mathcal{B}_{2}$, the function $\mathcal{A}_{n}g:B_{n}^{\mathcal{B}_{1}}\rightarrow B_{n}^{\mathcal{B}_{2}}$ defined by $\mathcal{A}_{n}g(z)_{[i]}=g(z_{[i]})$, for every $i\in\{1, \ldots  , n+1\}$ and snapshot $z$, is indeed a morphism in $\textbf{RSwap}_{C_{n}}$. If we take a second homomorphim $k:\mathcal{B}_{2}\rightarrow\mathcal{B}_{3}$, $\mathcal{A}_{n}(k\circ g)(z)_{[i]}=k\circ g(z_{[i]})=k(\mathcal{A}_{n}g(z)_{[i]})=\mathcal{A}_{n}k(\mathcal{A}_{n}g(z))_{[i]}$, what leads to, if applied to all $1\leq i\leq n+1$, $\mathcal{A}_{n}(k\circ g)=\mathcal{A}_{n}k\circ\mathcal{A}_{n}g$. It is clear how, when applied to the identity homomorphism of $\mathcal{B}$, $\mathcal{A}_{n}$ returns the identity morphism of $\mathcal{A}_{C_{n}}^{\mathcal{B}}$.

Now, for the functor $\textbf{Boo}_{n}$: given a morphism $h:\mathcal{A}_{C_{n}}^{\mathcal{B}_{1}}\rightarrow\mathcal{A}_{C_{n}}^{\mathcal{B}_{2}}$, according to Theorem~\ref{morp._are_n-tuples_hom.} there exists a homomorphism $g:\mathcal{B}_{1}\rightarrow\mathcal{B}_{2}$ with $h(z)_{[i]}=g(z_{[i]})$, for any $i\in\{1, \ldots  , n+1\}$ and snapshot $z$, and so $\textbf{Boo}_{n}h(a)=h((a, {\sim}a, 1, \ldots  , 1))_{1}=g(a)$, which of course means $\textbf{Boo}_{n}h$ is indeed a morphism of $\textbf{BA}$. If we take a second morphism $l:\mathcal{A}_{C_{n}}^{\mathcal{B}_{2}}\rightarrow \mathcal{A}_{C_{n}}^{\mathcal{B}_{3}}$, and suppose $k:\mathcal{B}_{2}\rightarrow\mathcal{B}_{3}$ is the homomorphism with $l(w)_{[i]}=k(w_{[i]})$, for every $1\leq i\leq n+1$ and snapshot $z$ of $B_{n}^{\mathcal{B}_{2}}$, consider an element $a$ of $|\mathcal{B}_{1}|$. Then $\textbf{Boo}_{n}(l\circ h)(a)=l(h(\overline{a}))_{[1]}=k(h(\overline{a})_{[1]})=k(g(a))=\textbf{Boo}_{n}l\circ\textbf{Boo}_{n}h(a)$, where we denote $(a, {\sim}a, 1, \ldots  , 1)$ by $\overline{a}$. Proving that the identity morphism of $\mathcal{A}_{C_{n}}^{\mathcal{B}}$ is mapped by $\textbf{Boo}_{n}$ into the identity homomorphism of $\mathcal{B}$ is straightforward.
\end{proof}

\begin{theorem}
$\textbf{Boo}_{n}\circ\mathcal{A}_{n}=Id_{\textbf{BA}}$ and $\mathcal{A}_{n}\circ\textbf{Boo}_{n}=Id_{\textbf{RSwap}_{C_{n}}}$.
\end{theorem}

\begin{proof}
$\textbf{Boo}_{n}\circ\mathcal{A}_{n}$ is the identity on objects given that, while $\mathcal{A}_{n}$ takes $\mathcal{B}$ to $\mathcal{A}_{C_{n}}^{\mathcal{B}}$, $\textbf{Boo}_{n}$ takes $\mathcal{A}_{C_{n}}^{\mathcal{B}}$ back to $\mathcal{B}$. Regarding morphisms, given a homomorphism of Boolean algebras $g:\mathcal{B}_{1}\rightarrow\mathcal{B}_{2}$ and $a\in|\mathcal{B}_{1}|$, let us denote $\mathcal{A}_{n}g$ by $h$, and then $(\textbf{Boo}_{n}\circ\mathcal{A}_{n})g(a)=\textbf{Boo}_{n}h(a)=h((a, {\sim}a, 1, \ldots  , 1))_{[1]}=g(a)$.

To prove $\mathcal{A}_{n}\circ\textbf{Boo}_{n}$ is the identity of $\textbf{RSwap}_{C_{n}}$, we start by noticing that $\textbf{Boo}_{n}$ first takes $\mathcal{A}_{C_{n}}^{\mathcal{B}}$ to $\mathcal{B}$, which is then taken back by $\mathcal{A}_{n}$ to $\mathcal{A}_{C_{n}}^{\mathcal{B}}$, meaning we have the identity on objects. Given a morphism $h:\mathcal{A}_{C_{n}}^{\mathcal{B}_{1}}\rightarrow\mathcal{A}_{C_{n}}^{\mathcal{B}_{2}}$ of $\textbf{RSwap}_{C_{n}}$, we know that there exists a homomorphism $g:\mathcal{B}_{1}\rightarrow\mathcal{B}_{2}$ such that $h(z)_{[i]}=g(z_{[i]})$, for every $1\leq i\leq n+1$ and snapshot $z$ of $B_{n}^{\mathcal{B}_{1}}$, meaning that $\textbf{Boo}_{n}h=g$ and so $(\mathcal{A}_{n}\circ\textbf{Boo}_{n})h(z)=(g(z_{[1]}), \ldots  , g(z_{[n+1]}))$, which equals exactly $h(z)$.
\end{proof}

We have proved that $\textbf{BA}$ and $\textbf{RSwap}_{C_{n}}$ are isomorphic, and since the first is a very rich category, we may translate many of its properties to the category of restricted swap structures for $C_{n}$. To give a few examples, remember that every atomic and complete Boolean algebra is isomorphic to $\textbf{2}^{\kappa}$, for $\textbf{2}$ the two-valued Boolean algebra and $\kappa$ the number of atoms in our target algebra. Since every finite Boolean algebra is atomic and complete, and $\mathcal{RM}_{C_{n}}^{\mathcal{B}}$ is finite iff $\mathcal{B}$ is finite, we have the following.

\begin{coro}
Every finite $\mathcal{RM}_{C_{n}}^{\mathcal{B}}$ is isomorphic to a power of $\mathcal{RM}_{C_{n}}$.
\end{coro}

On a stronger note, we know that every Boolean algebra is isomorphic to a field of sets, that is, a subalgebra of a power set algebra. Every power set algebra is itself complete and atomic, so we find that every Boolean algebra is isomorphic to a subalgebra of a power of $\textbf{2}$. To translate this result into $\textbf{RSwap}_{C_{n}}$ we need only to consider what is a substructure in this context. More generally, given RNmatrices $\mathcal{M}=(\mathcal{A}, D, \mathcal{F})$ and $\mathcal{M}^{*}=(\mathcal{A}^{*}, D^{*}, \mathcal{F}^{*})$ over the signature $\Theta$, $\mathcal{M}$ is a subRNmatrix of $\mathcal{M}^{*}$ if the universe $A$ of $\mathcal{A}$ is contained in the universe $A^{*}$ of $\mathcal{A}^{*}$ and the inclusion $j:A\rightarrow A^{*}$ satisfies: (1)~it is a $\Theta$-homomorphism between $\mathcal{A}$ and $\mathcal{A}^{*}$; (2)~$D\subseteq D^{*}$; and (3)~for every $\nu\in\mathcal{F}$, $j\circ\nu\in\mathcal{F}^{*}$. Of course, if both $\mathcal{M}$ and $\mathcal{M}^{*}$ are in a category $\mathcal{C}$ of RNmatrices as we have previously defined them, then $\mathcal{M}$ is a subRNmatrix of $\mathcal{M}^{*}$ iff $A\subseteq A^{*}$ and the inclusion $j$ is a morphism of $\mathcal{C}$.

\begin{lemma}
$\mathcal{RM}_{C_{n}}^{\mathcal{B}_{1}}$ is a subRNmatrix of $\mathcal{RM}_{C_{n}}^{\mathcal{B}_{2}}$ iff $\mathcal{B}_{1}$ is a subalgebra of $\mathcal{B}_{2}$.
\end{lemma}

\begin{proof}
Suppose first that $\mathcal{B}_{1}$ is a subalgebra of $\mathcal{B}_{2}$. If $z=(z_{[1]}, \ldots  , z_{[n+1]})$ is a snapshot of $B_{n}^{\mathcal{B}_{1}}$, meaning that $z\in |\mathcal{B}_{1}|^{n+1}$ and $(\bigwedge_{i=1}^{k}z_{[i]})\vee z_{[k+1]}=1$ for every $1\leq k\leq n$, it is true that: $z\in|\mathcal{B}_{2}|^{n+1}$, since $|\mathcal{B}_{1}|\subseteq|\mathcal{B}_{2}|$; and $(\bigwedge_{i=1}^{k}z_{[i]})\vee z_{[k+1]}=1$, now in $\mathcal{B}_{2}$, for every $1\leq k\leq n$, given that the operations in $\mathcal{B}_{2}$, over elements which also lie in $\mathcal{B}_{1}$, are the same as the operations of $\mathcal{B}_{1}$. So we now may consider the inclusion $j:B_{n}^{\mathcal{B}_{1}}\rightarrow B_{n}^{\mathcal{B}_{2}}$. It is a morphism of $\textbf{RSwap}_{C_{n}}$ since, for an arbitrary snapshot $z$, it may be written as $j(z)=(i(z_{[1]}), \ldots  , i(z_{[n+1]}))$, for $i:\mathcal{B}_{1}\rightarrow\mathcal{B}_{2}$ the inclusion homomorphism.

Reciprocally, suppose that $\mathcal{RM}_{C_{n}}^{\mathcal{B}_{1}}$ is a subRNmatrix of $\mathcal{RM}_{C_{n}}^{\mathcal{B}_{2}}$. Since $B_{n}^{\mathcal{B}_{1}}\subseteq B_{n}^{\mathcal{B}_{2}}$, for any $a\in|\mathcal{B}_{1}|$ we have that $(a, {\sim}a, 1, \ldots  , 1)\in B_{n}^{\mathcal{B}_{2}}$, and therefore $|\mathcal{B}_{1}|\subseteq |\mathcal{B}_{2}|$, so that we may consider the inclusion $i:|\mathcal{B}_{1}|\rightarrow|\mathcal{B}_{2}|$. It is a homomorphism of Boolean algebras because $i(a)=j((a, {\sim}a, 1, \ldots  , 1))_{[1]}$, for any $a$ in $\mathcal{B}_{1}$ and $j:\mathcal{RM}_{C_{n}}^{\mathcal{B}_{1}}\rightarrow\mathcal{RM}_{C_{n}}^{\mathcal{B}_{2}}$ the inclusion morphism.
\end{proof}

\begin{coro}
Every restricted swap structure $\mathcal{RM}_{C_{n}}^{\mathcal{B}}$ is a subRNmatrix of a power of $\mathcal{RM}_{C_{n}}$.
\end{coro}


\section{Final remarks} \label{FinRem}

This paper extends the application of RNmatrices to the study of da Costa's hierarchy we started in \cite{ConTol}, motivated by the same reasoning behind swap structures (\cite{CC16}). This provides characterizing semantics for each $C_{n}$ but, more importantly, offers an extensive class of models for these logics. Furthermore, it would seem that the same generalization for arbitrary Boolean algebras that takes $\mathcal{RM}_{C_{n}}$ to $\mathcal{RM}_{C_{n}}^{\mathcal{B}}$ could offer characterizations, as well as classes of models, for other systems of difficult treatment, including ones we have already presented RNmatrices for, such as $\mbCcl$ or $\cila$; and others we have not addressed yet, specially paraconsistent systems but also modal ones and possibly others. We also start an algebraic analysis of these systems, which appear to have a rich inner structure and could lead to a better understanding of models for $C_{n}$ altogether.

But the relevance of our restricted swap structures for da Costa's hierarchy is really made explicit by our characterization of their category. The very notion of a category of RNmatrices seems fruitful, and already possess many nice properties, but there is no reason one should expect it to be as well-behaved as $\textbf{RSwap}_{C_{n}}$ is. The fact that the category of restricted swap structures for $C_{n}$ is actually isomorphic to the category of non-trivial Boolean algebras suggests either the construction of swap structures as $n$-tuples, or the characteristics of da Costa's calculi themselves, or both, have properties capable of enriching their respective category of RNmatrices. Because of this, we are then inspired to study the category of restricted swap structures for other logics, not only for their own sake but also to clarify this question.

It is important too to look at the many applications of $\textbf{RSwap}_{C_{n}}$ which seem possible: after all, it is possible to capture much of the attributes of an algebraic logic from the variety of algebras performing the algebraization of the system. It is well-known that the  systems belonging to da Costa's hierarchy are not algebraizable (\cite{Mortensen80, Lewin}), but they do have corresponding categories of models capable of characterizing them, which in addition are isomorphic to the variety (the category of non-trivial Boolean algebras) which algebraize classical propositional logic.

\paragraph{Acknowledgements.} 

The first author acknowledges support from  the  National Council for Scientific and Technological Development (CNPq), Brazil
under research grant 306530/2019-8. The second author was supported by a doctoral scholarship from CAPES, Brazil.

\bibliographystyle{plain}

\end{document}